\documentclass[11pt]{article}

\UseRawInputEncoding

\usepackage{amssymb,amsmath,amsthm,epsfig}
\usepackage{latexsym, enumerate}
\usepackage{eepic}
\usepackage{epic}
\usepackage{graphicx}
\usepackage{color}
\usepackage{ifpdf}
\usepackage{subfigure}
\usepackage{tikz}
\usepackage{dsfont}
\usepackage{multirow}
\usepackage{makecell}
\usepackage{algorithm}
\usepackage{bm}
\usepackage{multirow}

\usepackage{cite}
\usepackage{hyperref}
\usepackage{xcolor}
\hypersetup{
  colorlinks,
  linkcolor={blue!80!black},
  urlcolor={blue!80!black},
  citecolor={blue!80!black}
}
\usepackage[capitalize,noabbrev]{cleveref}

\theoremstyle{plain}
\newtheorem{lem}{Lemma}[section]
\newtheorem{thm}[lem]{Theorem}
\newtheorem{cor}[lem]{Corollary}

\theoremstyle{definition}
\newtheorem{defn}{Definition}[section]

\theoremstyle{remark}
\newtheorem{rem}{Remark}[section]
\newtheorem{exm}{Example}[section]

\begin{document}

\title{ \Large\bf Shape reconstructions by using plasmon resonances}

\author{
Ming-Hui Ding\thanks{School of Mathematics, Hunan University, Changsha 410082,  China.\ \  Email: minghuiding@hnu.edu.cn}
\and
Hongyu Liu\thanks{Department of Mathematics, City University of Hong Kong, Kowloon, Hong Kong, China.\ \ Email: hongyu.liuip@gmail.com; hongyliu@cityu.edu.hk}
\and
Guang-Hui Zheng\thanks{School of Mathematics, Hunan University, Changsha 410082, China.\ \ Email: zhenggh2012@hnu.edu.cn}
}

\date{}
\maketitle

\begin{center}{\bf\large Abstract}
\end{center}\smallskip

We study the shape reconstruction of an inclusion from the {faraway} measurement of the associated electric field. This is an inverse problem of practical importance in biomedical imaging and is known to be notoriously ill-posed. By incorporating Drude's model of the permittivity parameter, we propose a novel reconstruction scheme by using the plasmon resonance with a significantly enhanced resonant field. We conduct a delicate sensitivity analysis to establish a sharp relationship between the sensitivity of the reconstruction and the plasmon resonance. It is shown that when plasmon resonance occurs, the sensitivity functional blows up and hence ensures a more robust and effective construction. Then we combine the Tikhonov regularization with the Laplace approximation to solve the inverse problem, which is an organic hybridization of the deterministic and stochastic methods and can quickly calculate the minimizer while capture the uncertainty of the solution. We conduct extensive numerical experiments to illustrate the promising features of the proposed reconstruction scheme.



\section{Introduction}
Plasmon resonance is the resonant oscillation of conduction electrons at the interface between negative and positive permittivity material stimulated by incident field. Plasmonics is revolutionizing many light-based technologies via electron oscillations in metals. We refer to  \cite{Jain2006, Anker2010, Raschke2003, Schultz2000, Nam2003, Baffou2010} and the references cited therein for many striking optical, phononic, biomedical, diagnostic and therapeutic applications in the physical literature. Recent studies have revealed the deep and intriguing connection between the plasmon resonance and the spectral study of the Neumann-Poincar\'e operator \cite{ACL,ACLS,ACKLM,DLZ, Mayergoyz2005, Grieser2014, Link2000, blasten2020}. In addition, there are many theoretical understandings and conceptual proposals about plasmonic devices.

 In \cite{Zhang2017}, by analyzing the imaginary part of the Green function, it is shown that one can achieve super-resolution and super-focusing by using plasmonic nanoparticles. In \cite{blasten2020,ACL,ACLS}, it is shown that the plasmon resonance concentrates and localises at high-curvature places, which can provide potential application in super-resolution imaging of plasmon particles.
 We would also like to mention in passing some related studies on plasmonic cloaking \cite{ACKLM,ando2016,DLZ, li2018,WN10,deng2020,li20181,li2019,LLL,bouchitte2010,Z}.
% , including the acoustics cloaking \cite{Zhang2017,ando2016,li2018}, elastic cloaking \cite{ando2018,deng2020,li20181,deng20201} and optical cloaking \cite{ammari20161,bouchitte2010,li2019}.
 In this paper, we study the shape reconstruction of an inclusion from the faraway measurement of the associated electric field. This is an inverse problem of practical importance in biomedical imaging and is known to be notoriously ill-posed. By incorporating Drude's model of the permittivity parameter, we propose a novel reconstruction scheme by using the plasmon resonance with a significantly enhanced resonant field.

 We next introduce the mathematical formulation of the inverse shape problem for our study. Let $D\subset\mathbb{R}^2$ be a bounded domain with a connected complement $\mathbb{R}^2\backslash\overline{D}$. Given a harmonic function $H$, we consider the following electrostatic problem:
 \begin{align}
\label{conducitivity pro}
\begin{cases}
&\nabla\cdot(\varepsilon\nabla u(x))=0\ \ \ \ \ \ \mathrm{in} \ \  \mathbb{R}^2,\\
&u(x)-H(x)=O(|x|^{-1})\ \ \mbox{as}\ \ |x|\rightarrow +\infty,
\end{cases}
\end{align}
where
\begin{align}
\label{conductivity}
\varepsilon=\varepsilon_{D}\chi(D)+\varepsilon_m\chi(\mathbb{{R}}^2\backslash \overline{{D}}),
\end{align}
and $\chi$ is the characteristic function. \eqref{conducitivity pro}-\eqref{conductivity} describes the transverse electromagnetic propagation in the quasi-static regime. $u$ signifies the transverse electric field and $\varepsilon$ signifies the permittivity parameter of the medium. Throughout, we shall assume that the background parameter $\varepsilon_m$ is a positive constant, whereas \textcolor[rgb]{0,0,0}{$\varepsilon_D$} is a complex-valued function of the illuminating frequency and fulfils the Drude's model. We shall supply more details about the Drude model in what follows. The shape reconstruction problem can be formulated as follows:

\emph{Inverse Problem (IP)}: Identify the shape of the inclusion, namely $\partial D$, from the measurement data $u^s=u-H$ on $\partial\Omega$ with $D\Subset\Omega$ associated with a fixed incident field $H$. For simplicity, we take $\Omega$ to be a central ball of radius $r\in\mathbb{R}_+$ with $r$ sufficiently large. Hence, the measurement represents the {far-field pattern} of the electric field.

The shape reconstruction problem introduced above is severely ill-posed and highly nonlinear. First, it is well known that due to the diffraction limit, the far field excited by the object carries information on a scale much larger than the operating wavelength, while information on a scale smaller than the operating wavelength is confined near the object itself. In addition, the scattering information in the quasi-static regime is very weak, and in the presence of measurement noise, the signal-to-noise ratio in the far field is low and signal distortion is serious \cite{ACL0}. We also refer to \cite{ACLS,ACLZ,H2,hintermuller2015,GLWZ,CDHLW,BL1,BL2,BL3,DCL,LT,LTY,CDLZ,YYL} for related studies in the literature on this inverse shape problem.

In this article, we first perform a shape sensitivity analysis and derive the shape sensitivity functional with respect to domain perturbation by a delicate asymptotic analysis. We establish the spectral expansion of the shape sensitivity functional, from which we can conclude the sharp relationship between the reconstruction sensitivity and the plasmon resonance. It indicates that the plasmon resonant field can render a more robust and effective reconstruction. Moreover, in order to overcome the ill-posedness, we combine the Tikhonov regularization method with the Laplace approximation (LA) to solve the inverse problem. This  hybrid method is essentially the organic combination of the deterministic method and stochastic method, {which can rapidly calculate the minimizer (Maximum a posteriori estimation point (MAP))} and capture statistical information of the solution more effectively. To provide a global view of our study, the major contributions of this work can be summarised as follows.

\begin{enumerate}
\item By using the layer-potential perturbation technique, we rigorously derive the asymptotic expansion of the perturbed far field with respect to the shape perturbation. Furthermore, we obtain the representation formula of the shape sensitivity functional.

\item Based on the spectral theory of the \textcolor[rgb]{0,0,0}{Neumann-poincar\'e operator}, we establish the delicate spectral expansion of the shape sensitivity functional. It indicates that when plasmon resonance occurs, the shape sensitivity can be improved dramatically.

\item \textcolor[rgb]{0,0,0}{Due to the severe ill-posedness of inverse problem, we use plasmon resonance to enhance the sensitivity, and then combine the Tikhonov regularization method with the Laplace approximation to solve the inverse problem and quantify the uncertainty of the solution.} Compared with the standard method, our numerical results show that the proposed method can significantly improve the accuracy and robustness of the numerical reconstruction.
\end{enumerate}

The rest of the paper is organized as follows. In Section 2, we provide preliminary knowledge on layer potential operators and plasmon resonance. In Section 3, we conduct the sensitivity analysis for the perturbed domain, and derive the spectral expansion of the shape sensitivity functional. In Section 4, we discuss the combination of the Tikhonov regularization method and the Laplace approximation. Sections 5 and 6 are respectively devoted to numerical experiments and conclusion.

\section{Preliminaries}
\subsection{Layer potentials and Neumann-Poincar\'e operator}

We collect a number of preliminary results on the layer potentials, in particular the Neumann-Poincar\'e operator for our subsequent use. Throughout this paper, we consider a domain $D$ with a $C^2$ boundary. The $L^2$ inner product and the corresponding norm on $\partial D$ are denoted by $\langle\cdot,\cdot\rangle$ and $\|\cdot\|$ in short, respectively. \textcolor[rgb]{0,0,0}{The single layer potential $\mathcal{S}_{D}$ and double layer potential $\mathcal{D}_{D}$  associated with $D$ are given by
\begin{align*}
\mathcal{S}_D[\varphi](x)&:=\int_{\partial D}\Gamma(x,y)\varphi(y)d\sigma(y),\ \ \ x\in \mathbb{R}^2,\\
\mathcal{D}_D[\varphi](x)&:=\int_{\partial D}\frac{\partial\Gamma(x,y)}{\partial \nu(y)}\varphi(y)d\sigma(y),\ \ \ x\in \mathbb{R}^2\setminus \partial D,
\end{align*}}
where $\varphi\in L^2(\partial D)$ is the density function, and the Green function $\Gamma(x,y)$ to the Laplacian in  $\mathbb{R}^2$ is given by
\begin{align*}
%\label{Gamma}
\Gamma(x,y)=\frac{1}{2\pi}\ln|x-y|.
\end{align*}
The notations $u|_\pm$ and $\frac{\partial u}{\partial \nu}|_\pm$ denote the traces on $\partial D$ from the outside and inside of $D$, respectively. \textcolor[rgb]{0,0,0}{The following jump relations hold \cite{ACLZ,ACKLM}}:
\begin{align*}
%\label{Jump rela}
\frac{\partial\mathcal{S}_D[\varphi]}{\partial \nu}\bigg|_{\pm}(x)&=\left(\pm\frac{1}{2}I+\mathcal{K}^*_D\right)[\varphi](x), \ \ x\in \partial D,
\end{align*}
where $\mathcal{K}^*_{D}$ is known as the Neumann-Poincar\'e (NP) operator defined by
\begin{align*}
%\label{NP}
\mathcal{K}_D^*[\varphi](x)=\frac{1}{2\pi}\int_{\partial D}\frac{\langle x-y,\nu(x)\rangle}{| x-y|^2}\varphi(y)d\sigma(y).
\end{align*}

Next, we recall some useful facts about the NP operator $\mathcal{K}^*_{ D}$ \cite{ACKLM, DLZ, Mayergoyz2005}.
%\textcolor[rgb]{0.00,0.00,1.00}{
% Let $H^{\frac{1}{2}}(\partial D)$ be the Sobolev space and $H^{-\frac{1}{2}}(\partial D)$ be its dual space with respect to the $L^2$-pairing $(\cdot,\cdot)_{-\frac{1}{2},\frac{1}{2}}$. We denote all set of $\varphi\in H^{-\frac{1}{2}}(\partial D)$ such that $(\varphi,1)_{-\frac{1}{2},\frac{1}{2}}=0$ by $H_0^{-\frac{1}{2}}$.}

\begin{lem}%\label{K basic}
(i) The Calder$\acute{o}$n identity holds: $\mathcal{S}_{ D} \mathcal{K}_{ D}^*=\mathcal{K}_{ D}\mathcal{S}_{ D}$ on $H_0^{-\frac{1}{2}}$, where $H_0^{-\frac{1}{2}}$ is the zero mean subspace of $H^{-\frac{1}{2}}$;

(ii) The operator $\mathcal{K}_{ D}^*$ is compact and self-adjoint in the Hilbert space $H_0^{-\frac{1}{2}}$ equipped with the following inner product
\begin{align*}
%\label{inner product}
\langle \varphi,\psi\rangle_{\mathcal{H}^*(\partial D)}=-\langle \mathcal{S}_{ D}[\psi],\varphi\rangle_{\frac{1}{2},-\frac{1}{2}},
\end{align*}
with $\langle\cdot,\cdot\rangle_{-\frac{1}{2},\frac{1}{2}}$ being the duality pairing between $H^{-1/2}(\partial D)$ and $H^{1/2}(\partial D)$

(iii) Let $\mathcal{H}^*(\partial D)$ be the space $H_0^{-\frac{1}{2}}(\partial D)$ with the new inner product in (ii). Let $(\lambda_j,\varphi_j), j=0,1,2,...$ be the eigenvalue and normalized eigenfunction pair of $\mathcal{K}_D^*$ in $\mathcal{H}^*(\partial D)$, then $\lambda_j\in(-\frac{1}{2},\frac{1}{2})$ and $\lambda_j\rightarrow 0$ as $j\rightarrow\infty$;

(iv) The following representation formula holds: for any $\varphi\in \mathcal{H}^{*}(\partial D)$,
\begin{align}
\label{spectral dec}
\mathcal{K}^*_{ D}[\varphi]=\sum_{j=1}^{\infty}\lambda_j\langle\varphi,\varphi_j\rangle_{\mathcal{H}^*(\partial D)}\varphi_j.
\end{align}
\end{lem}

\subsection{Plasmon resonance}
We next briefly discuss the mathematical framework of plasmon resonance. We first give the form of the solution of equation ($\ref{conducitivity pro})$. From \cite{ACL0,ACKLM,DLZ}, we have
\begin{align*}
%\label{solution}
u(x)=H(x)+\mathcal{S}_D[\phi](x),\ \ \  \mathrm{for}\ x \in \mathbb{{R}}^2,
\end{align*}
where $\phi\in L_0^2(\partial D):=\{\phi\in L^2(\partial D); \int_{\partial D}\phi=0\}$ satisfies
\begin{align}
\label{density}
(\lambda I-\mathcal{K}^*_D)[\phi](x)=\frac{\partial H}{\partial \nu}\bigg{|}_{\partial D},\ \ x\in\partial D,
\end{align}
with $\lambda$ given by
\begin{align}
\label{lambda}
\lambda=\frac{\varepsilon_D+\varepsilon_m}{2(\varepsilon_D-\varepsilon_m)}.
\end{align}

The permittivities of plasmon materials, such as noble metals, are different from the ordinary materials and may possess negative real parts. In fact, the electric permittivity $\varepsilon_D$ of the plasmon material is changing with respect to the operating frequency $\omega$. The $\varepsilon_D$ can be described by the Drude's model (see \cite{Challener2008,ADP,fang2015}),
\begin{align}
\label{Drude}
\varepsilon_D=\varepsilon_D(\omega)=\varepsilon_0\left(1-\frac{\omega_p^2}{\omega(\omega+i\gamma)}\right),
\end{align}
where $\varepsilon_0$ is the electric permittivity of the vacuum and is given by $\varepsilon_0=9\cdot10^{-12}F/m$. The plasmon frequency $\omega_p$ and the damping parameter $\gamma$ are strictly positive. When $\omega<\omega_p$ the real part of $\varepsilon_D$ can be negative. In particular, for the gold, the value of the plasmon frequency and the damping parameter are $\omega_p=2\cdot10^{15}s^{-1}$ and $\gamma=10^{14}s$, respectively. Let the operating frequency $\omega=6\cdot10^{14}Hz$ (visible light frequency), and from Drude's model (\ref{Drude}), the electric permittivity of gold nanoparticle is calculated as $\varepsilon_D\approx(-9.8108 + 1.8018i)\varepsilon_0$.

Now by applying the spectral decomposition of the $\mathcal{K}^*_D$ to the integral equation $(\ref{density})$, the density $\phi\in L_0^2(\partial D)$ becomes
\begin{align}
\label{density1}
\phi=\sum_{j=1}^\infty\frac{\langle\frac{\partial H}{\partial \nu},\varphi_j\rangle_{\mathcal{H}^*(\partial D)}}{\lambda-\lambda_j}\varphi_j,
\end{align}
where $\lambda_j$ are eigenvalues of $\mathcal{K}^*_D$ and they satisfy $|\lambda_j|<\frac{1}{2}$.
When the real part of $\varepsilon_D(\omega)$ is negative, it holds that
$|\mathfrak{Re}(\lambda(\omega))<\frac{1}{2}|$, where and also in what follows $\mathfrak{Re}(\lambda(\omega))$ signifies the real part of $\lambda(\omega)$.
The frequency $\omega$ is called a plasmon resonance frequency if it satisfies
\begin{align}\label{sp-cond}
\mathfrak{Re}(\lambda(\omega))=\lambda_j,
\end{align}
for some $j$, where $\lambda_j$ is an eigenvalue of the NP operator $\mathcal{K}^*_D$. For this reason, \eqref{sp-cond} is called the plasmon resonance condition. It is emphasized that only when the imaginary part $\mathfrak{Im}(\lambda(\omega))=0$ (lossless metal), the resonance condition (\ref{sp-cond}) can imply the density $\phi$ in (\ref{density1}) blow up.

In (\ref{density1}) the density $\phi$ will be amplified when the plasmon resonance condition is reached
provided that $\langle\frac{\partial H}{\partial\nu},\varphi_j\rangle_{\mathcal{H}^*(\partial D)}$ is nonzero.  As a result, the \textcolor[rgb]{0,0,0}{ $j$-th mode} of the far field $u-H$ will show a resonant behavior, and it is called that plasmon resonance occurs (see also \cite{DLZ}). As an illustration, we consider the cylindrical metal nanorod as an example. When the external electric field lines are parallel to the circular cross-section of the cylindrical metal nanorod, the physical process can be described by the two-dimensional mathematical model (\ref{conducitivity pro})--(\ref{conductivity}). Next, we set $\varepsilon_m=\varepsilon_0$ (vacuum) and notice that, as $D$ is a disk, the spectrum of $\mathcal{K}_D^*$ are $\{0,\frac{1}{2}\}$. If $\lambda_j=0$, from the plasmon resonance condition (\ref{sp-cond}) and the fact that $\gamma$ is sufficiently small in (\ref{Drude}), we nearly get $\mathfrak{Re}(\varepsilon_D(\omega))=-\varepsilon_m$. This relationship is called the Fr\"{o}hlich condition (see \cite{nicholls2018}). Furthermore, by Drude's model (\ref{Drude}), we can obtain the resonance frequency $\omega=\sqrt{\frac{\omega_p^2}{2}-\gamma^2}$. When $\gamma=0$ (lossless metal), the resonance frequency $\omega=\frac{\omega_p}{\sqrt{2}}$ or $f=\frac{\omega_p}{2\pi\sqrt{2}}$, which is called the Fr\"{o}hlich frequency (see \cite{Challener2008}).

\section{Shape sensitivity analysis}
\subsection{Sensitivity analysis for the perturbed domain}
We consider the sensitivity analysis for the shape reconstruction problem with a perturbed domain, namely evaluating the effect
of the domain variations on the {far-field} measurement data.
 %by using asymptotic expansions of the boundary perturbations of quasi-static approximation model resulting from small perturbations of the shape of small inclusion, we investigate the shape derivative of the forward operator $\mathcal{F}$ which is necessary for construction of algorithm to solve the minimization problem associated with the inverse problem.

Define the forward operator  $\mathcal{F}: X \rightarrow H^{\frac{1}{2}}(\partial \Omega)$ on a subset $X\subset C^2(\partial D) $:
\begin{align*}
%\label{forward oper}
\mathcal{F}(\partial D)=u^s|_{\partial \Omega},\ \ \ x \in X,
\end{align*}
where $u^s|_{\partial \Omega}=(u-H)|_{\partial \Omega}$ is the {far-field} measurement associated with \eqref{conducitivity pro} on the boundary $\partial \Omega$.

For a small $\epsilon\in\mathbb{R}_+$, we let $\partial D_\epsilon$ be an $\epsilon$ -perturbation of $D$, i.e.,
\begin{align*}
%\label{pertur-D}
\partial D_\epsilon:=\{\tilde{x}=x+\epsilon h(x)\nu(x), x\in \partial D\},
\end{align*}
where $h\in C^1(\partial D)$, and $\nu$ is the outward unit normal vector to $\partial D$. The solution $u_\epsilon$ to ($\ref{conducitivity pro}$) with $D_\epsilon$ has the following representation formula
\begin{align}
\label{per solution}
u_\epsilon=H(x)+\mathcal{S}_{D_\epsilon}[\tilde{\phi}_\epsilon](x),\ \ x\in \mathbb{R}^2,
\end{align}
where the density function $\tilde{\phi}_\epsilon$ is the solution to
\begin{align}
\label{per density}
(\lambda I-\mathcal{K}_{D_\epsilon}^*)[\tilde{\phi}_\epsilon](\tilde{x})=\frac{\partial H}{\partial\tilde{\nu}}\bigg{|}_{\partial D_\epsilon}\ \ \ \tilde{x}\in \partial D_\epsilon.
\end{align}
Let $\Psi_\epsilon$ be the diffeomorphism from $\partial D$ to $\partial D_\epsilon$ given by
\[
\Psi_\epsilon(x)=x+\epsilon h(x)\nu(x),\ \ \ x\in \partial D.
\]
Moreover, we denote $\tilde{\nu}$ the outward unit normal vector to  $\partial D_\epsilon$ and $d\tilde{\sigma}$ the line element of $\partial D_\epsilon$. The following expansions of $\tilde{\nu}$ and $d\tilde{\sigma}$ hold \cite{Lim2012}:
\begin{align}
\label{Per nu}
\tilde{\nu}(\tilde{x})&=\nu(x)-\epsilon h'(x)T(x)+O(\epsilon^2),\\
\label{Per line}
 d\tilde{\sigma}(\tilde{x})&=d\sigma(x)-\epsilon\tau(x)h(x)d\sigma(x)+O(\epsilon^2).
\end{align}
Here and throughout the rest of the paper, $\tau(x)$ signifies the curvature of $\partial D$ at $x$, $T$ is the unit tangential vector to $\partial D$, and $h'$ is the tangential derivative of $h$ on $\partial D$, i.e., $h'=\frac{\partial h}{\partial T}$, where $\frac{\partial}{\partial\tilde{\nu}}$ denotes the outward normal derivative on $\partial D_\epsilon$.

In view of ($\ref{per solution}$) and (\ref{per density}), in order to obtain the asymptotic formula of the perturbed far field $u_\epsilon^s|_{\partial \Omega}:=\mathcal{F}(\partial D_\epsilon)$, we need to obtain the corresponding asymptotic expansions of the operators $\mathcal{K}^*_{D_{\epsilon}}$ and $ \tilde{\phi}_\epsilon$.  The following lemmas can be found in \cite{Lim2012}.
\begin{lem}
 For $\tilde{\phi}\in L^2(\partial D_\epsilon)$, let $\phi:=\tilde{\phi}\circ\Psi_\epsilon$. Then there exists a constant $C$ depending only on the $C^2$-norm of $\partial D$ and $\parallel h\parallel_{C^1}$ such that
\[
\parallel\big{(}\mathcal{K}_{ D_\epsilon}^{*}[\tilde{\phi}]\big{)}\circ\Psi_\epsilon-\mathcal{K}_{ D}^{*}[{\phi}]-\epsilon \mathcal{K}_{ D}^{(1)}[\phi]\parallel_{L^2(\partial D)}\leq C\epsilon^2\parallel \phi\parallel_{L^2(\partial D)},
\]
with the operator $\mathcal{K}_D^{(1)}$ defined for any $\phi\in L^2(\partial D)$ by
\begin{align}\label{k1}
\mathcal{K}_{D}^{(1)}[\phi](x)=\mathrm{p.v.}\int_{\partial D}K_1(x,y)\phi(y)d\sigma(y), \ \ \ x\in \partial D,
\end{align}
where
\begin{align*}
K_1(x,y)&=-2\frac{\langle x-y,\nu(x)\rangle\langle x-y, h(x)\nu(x)-h(y)\nu(y)\rangle}{|x-y|^4}\\
&+\frac{\langle h(x)\nu(x)-h(y)\nu(y),\nu(x) \rangle}{|x-y|^2}-\frac{\langle x-y,\tau(x)h(x)\nu(x)+h'(x)T(x) \rangle}{|x-y|^2}\\
&+\frac{\langle x-y,\nu(x) \rangle}{|x-y|^2}(h(x)\tau(x)-h(y)\nu(y)).\\
\end{align*}
Here, $\mathrm{p.v.}$ stands for the Cauchy principal value.
\end{lem}
In fact, we can rewrite the operator $\mathcal{K}_{D}^{(1)}$  in terms of more familiar operators as follows,
%For $x, y\in \partial D(x\neq y)$, we have
%\begin{eqnarray}
%\small
%\begin{split}
%\frac{\partial}{\partial T(x)}\Gamma(x-y)&=\frac{1}{2\pi}\frac{\langle x-y,T(x)\rangle}{\mid x-y\mid^2},\\
%\frac{\partial^2}{\partial T(x)^2}\Gamma(x-y)&=\frac{1}{2\pi}\bigg{[}\frac{1}{\mid x-y\mid^2}+\frac{\langle x-y,\nu(x)\rangle\tau(x)}{\mid x-y\mid^2}-\frac{2\langle x-y,T(x)\rangle^2}{\mid x-y\mid^4}\bigg{]}\\
%&=\frac{1}{2\pi}\bigg{[}-\frac{1}{\mid x-y\mid^2}+\frac{\langle x-y,\nu(x)\rangle\tau(x)}{\mid x-y\mid^2}+\frac{2\langle x-y,\nu(x)\rangle^2}{\mid x-y\mid^4}\bigg{]}\\
%\end{split}
%\end{eqnarray}
%and
%\begin{eqnarray}
%\small
%\begin{split}
%\frac{\partial^2}{\partial \nu(x)\nu(y)}\Gamma(x-y)&=\frac{1}{2\pi}\bigg{[}-\frac{\langle\nu(x),\nu(y)\rangle}{\mid x-y\mid^2}+\frac{2\langle x-y,\nu(x)\rangle\langle x-y,\nu(y)\rangle}{\mid x-y\mid^4}\bigg{]}
%\end{split}
%\end{eqnarray}
%It then follows that
%\begin{eqnarray}
%\small
%\begin{split}
%&-\frac{\partial}{\partial T(x)}\bigg{(}h(x)\frac{\partial}{\partial T(x)}\bigg{)}\Gamma(x-y)+h(y)\frac{\partial^2}{\partial\nu(x)\nu(y)}\Gamma(x-y)\\
%&=-{2}\frac{\langle x-y,\nu(x)\rangle\langle x-y,h(x)\nu(x)-h(y)\nu(y)\rangle}{\mid x-y\mid^4}+\frac{\langle h(x)\nu(x)-h(y)\nu(y),\nu(x)\rangle}{\mid x-y\mid^2}\\
%&-\frac{\langle x-y,\tau(x)h(x)\nu(x)-h'(t)T(x)\rangle}{\mid x-y\mid^2}\\
%&=K_1(x,y)-\frac{\langle x-y,\nu(x)\rangle}{\mid x-y\mid^2}\big{(}h(x)\tau(x)-h(y)\tau(y)\big{)}
%\end{split}
%\end{eqnarray}
%If $\phi\in C^{1}(\partial D)$, then we get
\begin{align}
\small
\label{KD1}
\mathcal{K}_{D}^{(1)}[\phi]=-\frac{\partial}{\partial T}\bigg{(} h\frac{\partial \mathcal{S}_D[\phi]}{\partial T}\bigg{)}+\frac{\partial \mathcal{D}_D[h\phi]}{\partial \nu}+h\tau\mathcal{K}_{D}^*[\phi]-\mathcal{K}_D^*[h\tau \phi].
\end{align}

\begin{lem}\label{density asy}
 Let $\tilde{\phi}_\epsilon=(\lambda I-\mathcal{K}_{D_\epsilon}^*)^{-1}[\tilde{\nu}\cdot\nabla H]$, $\phi_\epsilon=\tilde{\phi}_\epsilon\circ\Psi_\epsilon$, and $\phi=(\lambda I-\mathcal{K}^*_{D})^{-1}[\nu\cdot\nabla H]$. Then we have
\begin{align*}
\parallel \phi_\epsilon-\phi-\epsilon\phi^{(1)}\parallel_{L^2(\partial D)}\leq C\epsilon^2\parallel \phi\parallel_{L^2(\partial D)},
\end{align*}
where $C$ is a constant depending only on the $C^2$-norm of $\partial D$ and $\parallel h\parallel_{C^1}$ and
\begin{align}
\label{asy density1}
\phi^{(1)}=(\lambda I-\mathcal{K}_D^*)^{-1}\bigg{(}h\langle(\nabla^2 H)\nu,\nu\rangle-h'\langle\nabla H,T\rangle+\mathcal{K}_D^{(1)} \phi\bigg{)}.
\end{align}
\end{lem}
Furthermore, for $\frac{\partial H}{\partial \tilde{\nu}}$ of (\ref{per density}), we can further obtain by using ($\ref{Per nu}$), that
\begin{align}
\label{per H/v}
\frac{\partial H}{\partial \tilde{\nu}}&=\nabla H(\Psi(x))\cdot \tilde{\nu}(\Psi(x))
=\frac{\partial H}{\partial \nu}(x)+\epsilon G^{(1)}(x)+O(\epsilon^2),\ \ \ x \in\partial D
\end{align}
where
\begin{align}
\label{G1}
G^{(1)}=h(x)\langle \nabla^2 H(x) \nu(x),\nu(x)\rangle-h'(x)\langle \nabla H(x), T(x)\rangle.
\end{align}

The asymptotic expansion of $\mathcal{F}(\partial D_\epsilon)$ under small perturbations of the boundary $\partial D$ in terms of the asymptotic parameter $\epsilon$ is given in the following theorem. We would like to point out that its proof is adapted from those in \cite{Feng2017, Lim2012}.
\begin{thm}\label{shape der thm}
Suppose that $\phi, \phi_d$ satisfy
\begin{align}
\label{condtion equation1}
(\lambda I-\mathcal{K}_{ D}^*)[\phi](y)&= \frac{\partial H}{\partial\nu}\bigg{|}_{\partial D},\\
\label{condtion equation2}
(\lambda I-\mathcal{K}_{D})[\phi_d](x,y)&=\Gamma(x,y),\ \ \ x\in\partial\Omega,\ y\in\partial D,
\end{align}
the following asymptotic expansion holds:
\begin{align*}
%\label{shape der}
u_\epsilon^s(x)|_{\partial \Omega}-u^s(x)|_{\partial \Omega}=\epsilon\langle h(y),P(x,y)\rangle_{L^2(\partial D)}+O(\epsilon^2),\ \ \ x\in\partial\Omega,\ y\in\partial D,
\end{align*}
where
\begin{align}\label{p1}
P(x,y)=\frac{\partial \phi_d(x,y)}{\partial T(y)}\bigg{(}\frac{\partial(H(y)+\mathcal{S}_{ D}[\phi](y))}{\partial T(y)}\bigg{)}+\bigg{(}\frac{\partial(\mathcal{D}_{D}[\phi_d](x,y)+\Gamma(x,y))}{\partial \nu(y)}\bigg{)}\phi(y).
\end{align}
\end{thm}

\begin{proof} From the form of the solution $(\ref{per solution})$, it holds that
\[
\mathcal{F}(\partial D_\epsilon)=\int_{\partial D_\epsilon}\Gamma(x,\tilde{y})\tilde{\phi}_\epsilon(\tilde{y})d\tilde{\sigma}(\tilde{y}),\ x\in \partial\Omega,
\]
where $\tilde{\phi}_\epsilon$ is solution to (\ref{per density}), $d\tilde{\sigma}(\tilde{y})$ has an expansion as (\ref{Per line}), and $\Gamma(x,\tilde{y})$ has the Taylor expansion as follows
\[
\Gamma(x,y+\epsilon h(y)\nu(y))=\Gamma(x,y)+\epsilon h(y)\frac{\partial \Gamma(x,y)}{\partial\nu(y)}+O(\epsilon^2)\ \ \ y\in \partial D,\ x\in \partial \Omega.
\]
By virtue of Lemma \ref{density asy}, we can obtain that
\begin{align*}
\mathcal{F}(\partial D_\epsilon)-\mathcal{F}(\partial D)=&\epsilon\int_{\partial D}\Gamma(x,y) \phi^{(1)}(y)d\sigma(y)+\epsilon\int_{\partial D}\bigg{(}\frac{\partial \Gamma(x,y)}{\partial \nu(y)}-\tau(y) \Gamma(x,y)\bigg{)} \phi(y) h(y) d\sigma(y)\\
&+O(\epsilon^2).
\end{align*}
Next, we first calculate the term $\int_{\partial D}\Gamma(x,y) \phi^{(1)}(y)d\sigma(y)$. From $(\ref{asy density1})$ and (\ref{condtion equation2}) we have
\begin{align*}
\small
\int_{\partial D}\Gamma(x,y)\phi^{(1)}(y)d\sigma(y)=&\int_{ \partial D}(\lambda I-\mathcal{K}_{ D})\phi_d(x,y) \phi^{(1)}(y)d\sigma(y)\\
=&\int_{\partial D}\phi_d(x,y)(\lambda I-\mathcal{K}_{ D}^*) \phi^{(1)}(y)d\sigma(y)\\
=&\int_{\partial D}\phi_d(x,y)\bigg{(}h(y)\langle(\nabla^2 H(y))\nu(y),\nu(y)\rangle-h'(y)\frac{\partial H(y)}{\partial T(y)}\bigg{)}d\sigma(y)\\
&+\int_{\partial D}\phi_d(x,y)\mathcal{K}_{ D}^{(1)}[\phi](y)d\sigma(y)\\
=&:l_1+l_2.
\end{align*}
 We treat $l_1, l_2$ separately. Since
\begin{align}
\label{ThH}
\frac{\partial}{\partial T(y)}\bigg{(} \phi_d(x,y) \frac{\partial H(y)}{\partial  T(y)} \bigg{)}&=\frac{\partial \phi_d(x,y)}{\partial T(y)}\frac{\partial H(y)}{\partial T(y)}+\phi_d(x,y)\langle (\nabla^2 H(y))T(y),T(y)\rangle\\
&+\phi_d(x,y)\tau(y)\frac{\partial H(y)}{\partial \nu(y)}\nonumber,
\end{align}
and noting that $H(y)$ is a harmonic function and (\ref{ThH}), $l_1$ has the following form
\begin{align*}
\small
l_1=&\int_{\partial D}\phi_d(x,y) \bigg{(}h(y)\langle(\nabla^2 H(y))\nu(y),\nu(y)\rangle-h'(y)\langle\nabla H(y),T(y)\rangle\bigg{)}d\sigma(y)\\
=&\int_{\partial D}h(y)\bigg{(}\phi_d(x,y)\langle(\nabla^2 H(y))\nu(y),\nu(y)\rangle+\frac{\partial\phi_d(x,y)}{\partial T(y)}\frac{\partial H(y)}{\partial T(y)}\\&+\phi_d(x,y)\langle(\nabla^2 H(y))T(y),T(y)\rangle
+\phi_d(x,y)\tau(y)\frac{\partial H(y)}{\partial \nu(y)}\bigg{)} d\sigma(y)\\
=&\int_{\partial D}h(y)\bigg{(}\frac{\partial\phi_d(x,y)}{\partial T(y)}\frac{\partial H(y)}{\partial T(y)}+\phi_d(x,y)\Delta H(y)+\phi_d(x,y)\tau(y)\frac{\partial H(y)}{\partial \nu(y)}\bigg{)}d\sigma(y)\\
=&\int_{\partial D}h(y)\bigg{(}\frac{\partial\phi_d(x,y)}{\partial T(y)}\frac{\partial H(y)}{\partial T(y)}+\phi_d(x,y)\tau(y)\frac{\partial H(y)}{\partial \nu(y)}\bigg{)}d\sigma(y).
\end{align*}
Next, by ($\ref{KD1}$), we have
\allowdisplaybreaks
\begin{align}
l_2=&\int_{\partial D}\phi_d(x,y) \mathcal{K}_{D}^{(1)}[\phi](y) d\sigma(y)\nonumber\\
=&\int_{\partial D}\phi_d (x,y) \bigg{[}-\frac{\partial}{\partial T(y)}\bigg{(} h(y)\frac{\partial\mathcal{S}_D[\phi](y)}{\partial T(y)}\bigg{)}+\frac{\partial}{\partial \nu(y)}\mathcal{D}_{ D}[\phi h](y)+h(y)\tau(y)\mathcal{K}_{D}^*[\phi](y)\nonumber\\
&-\mathcal{K}_{ D}^*[h\tau\phi](y)\bigg{]}d\sigma(y)\nonumber\\
=&\int_{\partial D}h(y) \bigg{[}\frac{\partial\phi_d(x,y)}{\partial T(y)}\frac{\partial(\mathcal{S}_{D}\phi(y))}{\partial T(y)}+\frac{\partial}{\partial \nu(y)}\mathcal{D}_{ D}[\phi_d](x,y)\phi(y)+\phi_d(x,y)\tau(y)\mathcal{K}_{D}^*[\phi](y)\nonumber\\
&-\tau(y)\phi(y)\mathcal{K}_{ D}[\phi_d](x,y)\bigg{]}d\sigma(y),\nonumber
\end{align}
from which, together with the use of Proposition 4.1 in \cite{Lim2012}, we can further show that
\[
\int_{\partial D}\phi_d(x,y)\frac{\partial \mathcal{D}_D[h\phi](y)}{\partial \nu(y)}d\sigma(y)=\int_{\partial D}\frac{\partial \mathcal{D}_D[\phi_d](x,y)}{\partial \nu(y)}h(y)\phi(y) d\sigma(y).
\]
 Thus we obtain that
 \allowdisplaybreaks
\begin{align*}
\label{ 1}
\mathcal{F}(\partial D_\epsilon)-\mathcal{F}(\partial D) &=\int_{\partial D}h(y) \bigg{[}\frac{\partial \phi_d(x,y)}{\partial T(y)}\bigg{(}\frac{\partial(\mathcal{S}_{ D}[\phi])(y)}{\partial T(y)}\bigg{)}
+\frac{\partial(\mathcal{D}_{D}[\phi_d])(x,y)}{\partial \nu(y)}\phi(y)\\&+\tau(y) \phi_d(x,y)\mathcal{K}^*_{D}[\phi](y)-\tau(y)\mathcal{K}_D[\phi_d](x,y)\phi(y) \\
&+\phi_d(x,y)\tau(y)\frac{\partial H(y)}{\partial \nu(y)} +\frac{\partial \phi_d(x,y)}{\partial T(y)}\frac{\partial H(y)}{\partial T(y)}
\\&+\bigg{(}\frac{\partial \Gamma(x,y)}{\partial \nu(y)}-\tau(y) \Gamma(x,y)\bigg{)}\phi(y)\bigg{]}d\sigma(y)+O(\epsilon^2).
\end{align*}
From ($\ref{condtion equation1}$) and ($\ref{condtion equation2}$), we know that
\begin{align*}
\mathcal{K}^*_D[\phi](y)+\frac{\partial H(y)}{\partial \nu(y)}=\lambda\phi(y),\ \ \
\mathcal{K}_{D}[\phi_d](x,y)+\Gamma(x,y)=\lambda\phi_d(x,y).
\end{align*}
Thus  it follows that
\begin{align}
\mathcal{F}(\partial D_\epsilon)-\mathcal{F}(\partial D)&=\epsilon\int_{\partial D}h(y) \bigg{[}\frac{\partial \phi_d(x,y)}{\partial T(y)}\bigg{(}\frac{\partial(H(y)+\mathcal{S}_{ D}[\phi](y))}{\partial T(y)}\bigg{)}\nonumber\\
&+\bigg{(}\frac{\partial(\Gamma(x,y)+\mathcal{D}_{D}[\phi_d](x,y))}{\partial \nu(y)}\bigg{)}\phi(y)\bigg{]}d\sigma(y)+O(\epsilon^2)\nonumber,
\end{align}
which readily completes the proof.
\end{proof}

In order to investigate variations in the measurement resulting from variations in the shape of the underlying object, we
introduce the following definition of the shape sensitivity functional.
\begin{defn}\label{de1}
The shape sensitivity functional for the {far-field} measurement $u^s|_{\partial \Omega}$ with respect to the shape of $\partial D$ is defined as
\begin{align*}
%\label{ssf}
SSF(\partial D):=\lim_{\epsilon\rightarrow0}\frac{u_\epsilon^s|_{\partial \Omega}-u^s|_{\partial \Omega}}{\epsilon}.
\end{align*}
\end{defn}

\begin{rem}
From Definition \ref{de1}, it is easy to see that the shape sensitivity function is actually the shape derivative of the forward operator $\mathcal{F}(\partial D)$ (cf. \cite{Lim2012}). Furthermore, by using Theorem \ref{shape der thm}, the shape sensitivity function can be rewritten as
\begin{align*}
%\label{ssf1}
SSF(\partial D)=\langle h,P\rangle_{L^2(\partial D)},
\end{align*}
where $P$ is defined in (\ref{p1}).
\end{rem}
\subsection{Shape sensitivity analysis and plasmon resonance}
In this subsection, we shall derive the spectral representation of the shape sensitivity functional in this subsection. It indicates that, when the plasmon resonance occurs, the shape sensitivity functional will be amplified and exhibits a large peak. Hence, the plasmon resonance can be used to significantly increase the sensitivity of the {far-field} measurement with respect to the shape of the underlying domain. {For simplicity, in the subsequent spectral analysis, we always exclude the essential spectrum $0$ from the spectrum set of NP operator and assume the eigenvalues is simple.}

First, we derive the asymptotic formulas for the eigenvalues and eigenfunctions of the NP operator with respect to the asymptotic parameter $\epsilon$. 
%As for asymptotic formula of the eigenvalues, one can also find in \cite{ACLZ}.

\begin{lem}\label{per eig}
Suppose that $\{\lambda_{j,\epsilon}, \varphi_{j,\epsilon}\},j=0,1,2,\cdots,$ are the perturbed eigenvalues and eigenfunctions of $\mathcal{K}^*_{D_\epsilon}$. Then we have as $\epsilon\rightarrow 0$ that
\begin{align*}
\lambda_{j,\epsilon}=\lambda_j+\epsilon \lambda_{j,\epsilon}^{(1)}+O(\epsilon^2),\quad \varphi_{j,\epsilon}=\varphi_j+\epsilon \varphi_{j,\epsilon}^{(1)}+O(\epsilon^2),
\end{align*}
where
\begin{align*}
\lambda_{j,\epsilon}^{(1)}=\langle \mathcal{K}_D^{(1)}\varphi_j,\varphi_j\rangle_{\mathcal{H}^*(\partial D)},\quad
\varphi_{j,\epsilon}^{(1)}=\sum_{j\neq l}\frac{\langle \mathcal{K}_D^{(1)}\varphi_j,\varphi_l\rangle_{\mathcal{H}^*(\partial D)}}{\lambda_j-\lambda_l}\varphi_l,
\end{align*}
and $\mathcal{K}_D^{(1)}$ is defined by (\ref{k1}) or (\ref{KD1}).
\end{lem}

\begin{proof} Since $\mathcal{K}_{D_\epsilon}^*[\varphi_{j,\epsilon}]=\lambda_{j,\epsilon}\varphi_{j,\epsilon}$, and by Lemma 3.1, we can obtain the following result
\[
(\mathcal{K}_{D}^*+\epsilon\mathcal{K}_{D}^{(1)}+O(\epsilon^2))(\varphi_j+\epsilon \varphi_{j,\epsilon}^{(1)}+O(\epsilon^2))=(\lambda_j+\epsilon \lambda_{j,\epsilon}^{(1)}+O(\epsilon^2))(\varphi_j+\epsilon \varphi_{j,\epsilon}^{(1)}+O(\epsilon^2)),
\]
which implies that
\begin{align}
%\label{eig pertured}
\mathcal{K}_{D}^*\varphi_j&=\lambda_j\varphi_j\nonumber,\\
\label{eig pertured2}
\mathcal{K}_{D}^{(1)}\varphi_j+\mathcal{K}_{D}^*\varphi_{j,\epsilon}^{(1)}&=
\lambda_{j,\epsilon}^{(1)}\varphi_j+\lambda_j\varphi_{j,\epsilon}^{(1)}.
\end{align}
By inner producting $\varphi_l$  in both sides of ($\ref{eig pertured2}$), we further have
\begin{align*}
\langle\mathcal{K}_{D}^{(1)}\varphi_j,\varphi_l\rangle_{\mathcal{H}^*(\partial D)}+\langle \mathcal{K}_{D}^*\varphi_{j,\epsilon}^{(1)},\varphi_l\rangle_{\mathcal{H}^*(\partial D)}
=
\lambda_{j,\epsilon}^{(1)}\langle\varphi_j,\varphi_l\rangle_{\mathcal{H}^*(\partial D)}+\lambda_j\langle\varphi_{j,\epsilon}^{(1)},\varphi_l\rangle_{\mathcal{H}^*(\partial D)}.
\end{align*}
Since the operator $\mathcal{K}_D^*$ is self-adjoint, we have
\[
\langle \mathcal{K}_D^* \varphi_{j,\epsilon}^{(1)},\varphi_l\rangle_{\mathcal{H}^*(\partial D)}=\langle\varphi_{j,\epsilon}^{(1)},\mathcal{K}_D^*\varphi_l \rangle_{\mathcal{H}^*(\partial D)}=\lambda_l\langle \varphi_{j,\epsilon}^{(1)},\varphi_l\rangle_{\mathcal{H}^*(\partial D)},
\]
and thus
\[
\lambda_{j,\epsilon}^{(1)}=\langle \mathcal{K}_D^{(1)}\varphi_j,\varphi_j\rangle_{\mathcal{H}^*(\partial D)}.
\]
From ($\ref{eig pertured2}$), we have
\[
(\lambda_jI-\mathcal{K}_D^*)\varphi_{j,\epsilon}^{(1)}=\mathcal{K}_D^{(1)}\varphi_j-\lambda_{j,\epsilon}^{(1)}\varphi_j.
\]
Then it is straightforward to see that
\[
\varphi_{j,\epsilon}^{(1)}=\sum_{j\neq l}\frac{\langle \mathcal{K}_D^{(1)}\varphi_j,\varphi_l\rangle_{\mathcal{H}^*(\partial D)}}{\lambda_j-\lambda_l}\varphi_l.
\]

The proof is complete.
\end{proof}

{Let $\text{dist}(\lambda,\sigma(\mathcal{K}_D^*))$ be the distance of $\lambda$ to the spectrum set $\sigma(\mathcal{K}_D^*)$.} Next, we are ready to state our main result in the following theorem.
\begin{thm}\label{spectral}
As $\epsilon\rightarrow 0$, the perturbation $u^s_\epsilon|_{\partial\Omega}-u|_{\partial \Omega}$ has the following spectral expansion:
\begin{align*}
u^s_\epsilon(x)-u^s(x)=\epsilon \mathcal{T}(x)
{+O\left(\frac{\epsilon^2}{\mathrm{dist}(\lambda,\sigma(\mathcal{K}_D^*))^2}\right)},\ \ \ \ x\in\partial\Omega,
\end{align*}
with
\begin{align}
\label{ux}
\mathcal{T}(x)=&\sum_{j=1}^{\infty}\frac{\langle \frac{\partial H}{\partial \nu},\varphi_{j,\epsilon}^{(1)}\rangle+\langle G^{(1)},\varphi_j\rangle\mathcal{S}_D[\varphi_j](x)+\langle \frac{\partial H}{\partial \nu},\varphi_j\rangle \mathcal{Z}(x)}{\lambda-\lambda_j}\\&+
\sum_{j=1}^{\infty}\frac{\lambda_{j,\epsilon}^{(1)}\langle\frac{\partial H}{\partial \nu},\varphi_j\rangle\mathcal{S}_D[\varphi_j](x)}{(\lambda-\lambda_j)^2}\nonumber,
\end{align}
where $G^{(1)}$ is defined by (\ref{G1}), $\langle\cdot,\cdot\rangle$ signifies $\langle\cdot,\cdot\rangle_{\mathcal{H}^*(\partial D)}$, and
\begin{align*}
\small
\mathcal{Z}(x)=\int_{\partial D}
\left(\frac{\partial \Gamma(x,y)}{\partial \nu(y)}h(y)\varphi_j(y)+\Gamma(x,y)\varphi_{j,\epsilon}^{(1)}(y)-\tau(y) h(y)\Gamma(x,y)\varphi_j(y)\right)d\sigma(y).
\end{align*}
\end{thm}

\begin{proof} According to ($\ref{spectral dec}$), $\mathcal{F}(\partial D)$ can be decomposed as follow
\begin{align*}
\mathcal{F}(\partial D)&=\int_{\partial D}\Gamma(x,y)(\lambda I-\mathcal{K}_{D}^*)^{-1}\bigg{[}\frac{\partial H}{\partial \nu}\bigg{]}(y) d\sigma(y)\\
&=\sum_{j=1}^{\infty} \frac{\langle \frac{\partial H}{\partial \nu},\varphi_j\rangle}{\lambda-\lambda_j}\int_{\partial D}\Gamma(x,y)\varphi_j(y)d\sigma(y)\\
&=\sum_{j=1}^{\infty} \frac{\langle \frac{\partial H}{\partial \nu},\varphi_j\rangle\mathcal{S}_D[\varphi_j](x)}{\lambda-\lambda_j}.
\end{align*}
In a similar manner, $\mathcal{F}(\partial D_\epsilon)$ can be decomposed as
\begin{align*}
\small
\mathcal{F}(\partial D_\epsilon)&=\int_{\partial D_\epsilon}\Gamma(x,\tilde{y})(\lambda I-\mathcal{K}_{D_\epsilon}^*)^{-1}\bigg{[}\frac{\partial H}{\partial \tilde{\nu}}\bigg{]}(\tilde{y}) d\tilde{\sigma}(\tilde{y})\\
&=\sum_{j=1}^{\infty} \frac{\langle \frac{\partial H}{\partial \tilde{\nu}},\varphi_{j,\epsilon}\rangle\mathcal{S}_{D_{\epsilon}}[\varphi_{j,\epsilon}](x)}{\lambda-\lambda_{j,\epsilon}}.
\end{align*}
Furthermore, using the perturbation of eigenvalues and eigenfunctions in Lemma \ref{per eig}, we can show that
\begin{align}
\small
\label{per Sd}
\mathcal{S}_{D_{\epsilon}}[\varphi_{j,\epsilon}](x)=\int_{\partial D}\Gamma(x,y)\varphi_j(y) d\sigma(y)+\epsilon \mathcal{Z}(x)+O(\epsilon^2),
\end{align}
where
\begin{align*}
\small
\mathcal{Z}(x)=\int_{\partial D}
\bigg{(}\frac{\partial \Gamma(x,y)}{\partial \nu(y)}h(y)\varphi_j(y)+\Gamma(x,y)\varphi_{j,\epsilon}^{(1)}(y)-\tau(y) h(y)\Gamma(x,y)\varphi_j(y)\bigg{)}d\sigma(y).
\end{align*}
From $(\ref{per H/v})$, $(\ref{per Sd})$ and Lemma \ref{per eig},  we deduce that

\allowdisplaybreaks
\begin{small}
\begin{align*}
\mathcal{F}(\partial D_\epsilon)=&\sum_{j=1}^{\infty} \frac{\langle \frac{\partial H}{\partial \tilde{\nu}},\varphi_{j,\epsilon}\rangle\mathcal{S}_{D_\epsilon}[\varphi_{j,\epsilon}](x)} {\lambda-\lambda_{j,\epsilon}}\\
=&\sum_{j=1}^{\infty}\frac{\langle \frac{\partial H}{\partial \nu}+\epsilon G^{(1)}+O(\epsilon^2), \varphi_j+\epsilon \varphi_{j,\epsilon}^{(1)}+O(\epsilon^2)\rangle(\mathcal{S}_D[\varphi_j](x)+\epsilon \mathcal{Z}(x)+O(\epsilon^2))}{\lambda-\lambda_j-\epsilon\lambda_{j,\epsilon}^{(1)}+{O(\frac{\epsilon^2}{\lambda-\lambda_j})}}\\
=&\sum_{j=1}^{\infty}\frac{\langle \frac{\partial H}{\partial \nu},\varphi_j\rangle\mathcal{S}_D[\varphi_j](x)+\epsilon\bigg{(}\langle \frac{\partial H}{\partial \nu},\varphi_{j,\epsilon}^{(1)}\rangle\mathcal{S}_D[\varphi_j](x)+\langle G^{(1)},\varphi_j\rangle\mathcal{S}_D[\varphi_j]+\langle \frac{\partial H}{\partial \nu},\varphi_j\rangle \mathcal{Z}(x)\bigg{)}+O(\epsilon^2)}{(\lambda-\lambda_j)
(1-\frac{\epsilon\lambda_{j,\epsilon}^{(1)}}{\lambda-\lambda_j}+{O(\frac{\epsilon^2}{\lambda-\lambda_j}))}}\\
=&\bigg{(}\sum_{j=1}^{\infty} \langle \frac{\partial H}{\partial \nu},\varphi_j\rangle\mathcal{S}_D[\varphi_j](x)+\epsilon\big{(}\langle \frac{\partial H}{\partial \nu},\varphi_{j,\epsilon}^{(1)}\rangle\mathcal{S}_D[\varphi_j](x)+\langle G^{(1)},\varphi_j\rangle\mathcal{S}_D[\varphi_j](x)\\&+\langle \frac{\partial H}{\partial \nu},\varphi_j\rangle \mathcal{Z}(x)\big{)}+O\bigg{(}\epsilon^2\bigg{)}\bigg{)}
\cdot\bigg{(}\sum_{k=0}^{\infty}\frac{1}{(\lambda-\lambda_j)}\big{(}\frac{\epsilon\lambda_{j,\epsilon}^{(1)}}
{\lambda-\lambda_j}+{O(\frac{\epsilon^2}{\lambda-\lambda_j})\big{)}^k}\bigg{)}\\
=&\sum_{j=1}^{\infty}\frac{\langle\frac{\partial H}{\partial \nu},\varphi_j\rangle\mathcal{S}_D[\varphi_j](x)}{\lambda-\lambda_j}+\epsilon\bigg{(}\frac{\langle \frac{\partial H}{\partial \nu},\varphi_{j,\epsilon}^{(1)}\rangle\mathcal{S}_D[\varphi_j](x)+\langle G^{(1)},\varphi_j\rangle\mathcal{S}_D[\varphi_j](x)+\langle \frac{\partial H}{\partial \nu},\varphi_j\rangle \mathcal{Z}(x)}{\lambda-\lambda_j}\\&+\frac{\langle\frac{\partial H}{\partial\nu},\varphi_j\rangle\mathcal{S}_D[\varphi_j](x)\lambda_{j,\epsilon}^{(1)}}
{(\lambda-\lambda_j)^2}\bigg{)}{+O\left(\frac{\epsilon^2}{\text{dist}(\lambda,\sigma(\mathcal{K}_D^*))^2}\right)}.
\end{align*}
\end{small}
Therefore it follows that
\begin{align*}
\mathcal{F}(\partial D_\epsilon)-\mathcal{F}(\partial D)&=\epsilon\sum_{j=1}^{\infty}\frac{\langle \frac{\partial H}{\partial \nu},\varphi_{j,\epsilon}^{(1)}\rangle+\langle G^{(1)},\varphi_j\rangle\mathcal{S}_D[\varphi_j](x)+\langle \frac{\partial H}{\partial \nu},\varphi_j\rangle \mathcal{Z}(x)}{\lambda-\lambda_j}\\
&+\epsilon\sum_{j=1}^{\infty}\frac{\langle\frac{\partial H}{\partial \nu},\varphi_j\rangle\mathcal{S}_D[\varphi_j](x)
\lambda_{j,\epsilon}^{(1)}}{(\lambda-\lambda_j)^2}{+O\left(\frac{\epsilon^2}{\text{dist}(\lambda,\sigma(\mathcal{K}_D^*))^2}\right)}.
\end{align*}

The proof is complete.
\end{proof}

\begin{rem}\label{rem:32}
From (\ref{ux}), one readily sees that the \textcolor[rgb]{0,0,0}{$j$-th mode} in the expansion formula contains both $\mathcal{O}\left(\frac{1}{\lambda-\lambda_j}\right)$ and $\mathcal{O}\left(\frac{1}{(\lambda-\lambda_j)^2}\right)$. Thus, for the sufficiently small loss ($\mathfrak{Im}(\varepsilon_D)\rightarrow0$), the \textcolor[rgb]{0,0,0}{$j$-th mode} will exhibit a large peak if the plasmon resonance condition (\ref{sp-cond}) is fulfilled.
\end{rem}

From Theorem \ref{spectral}, we can straightforwardly obtain the spectral representation formula of the shape sensitivity function as follows.
\begin{cor}\label{cor3.6}
{If $\epsilon=o\left(\mathrm{dist}(\lambda,\sigma(\mathcal{K}_D^*))^2\right)$, (as $\mathrm{dist}(\lambda,\sigma(\mathcal{K}_D^*))\rightarrow0$),} the shape sensitivity function can be represented as
\begin{align*}
SSF(\partial D)=\mathcal{T}(x),
\end{align*}
where $\mathcal{T}(x)$ is defined by (\ref{ux}).
\end{cor}

\begin{rem}
Similar to Remark~\ref{rem:32}, one sees that when $\lambda=\frac{\varepsilon_D+\varepsilon_m}{2(\varepsilon_D-\varepsilon_m)}$ is very close to an eigenvalue $\lambda_j$ of the NP operator, i.e., the plasmon resonance occurs, the shape sensitivity function $SSF(\partial D)$ will be amplified dramatically for the $j$-th mode. Hence, the plasmon resonance can improve the sensitivity of the shape construction. Moreover, from the Drude's model (\ref{Drude}), we can compute the plasmon resonance frequency for improving the sensitivity.
\end{rem}

\begin{rem}
From a numerical point of view, the discrete version of the shape sensitivity functional is the shape sensitivity matrix. The sensitivity of the shape reconstruction can be analysed by the singular value decomposition of the sensitivity matrix. In Section 5, the numerical results shall show that when the plasmon resonance occurs, the singular values of the sensitivity matrix increase dramatically. That is the plasmon resonance technology can enhance the stability of the Gauss-Newton iteration algorithm that we use therein (see Section 4).
\end{rem}

\section{Tikhonov regularization and Laplace approximation}

In this section, we discuss the numerical issues for the shape reconstruction. First, a numerical implementation requires a parametrization of the boundary $\partial D$. Here we assume that $\partial D$ is starlike boundary curve with respect to the origin, i.e. there exists $q\in C^2[0,2\pi]$ such that
\begin{align*}
%\label{boundry}
\partial D=\{\vec{q}(t)=q(t)\bigg{(} \begin{array}{c}
\cos t\\
\sin t
\end{array} \bigg{)},\ t\in[0,2\pi]   \}.
\end{align*}
The admissible set $X=\{q\in C^2([0,2\pi]): q>0 \}$, and the  forward operator maps $X$ into $H^{\frac{1}{2}}(\partial \Omega)$. Without change of notation we write $\mathcal{F}(q)=u^s(x)|_{\partial\Omega}$.

In order to overcome the ill-posedness of the inverse problem, we apply the Tikhonov regularization method to tackle it. The corresponding variational functional is given as follows,
\begin{align}
\label{variational reg}
J[q]:=\frac{1}{2}\parallel \mathcal{F}(q)-u^{s,\delta}\parallel_2^2+\frac{\mu}{2}\parallel q \parallel_{L^2[0,2\pi]}^2,
\end{align}
where, $u^{s,\delta}=(u_1^{s,\delta},u_2^{s,\delta},\cdots,u_n^{s,\delta})$ signfies the discrete measurement data on $\partial\Omega$. Here, $\|\cdot\|_2$ denotes the $2$-norm in $\mathbb{R}^n$, and $\mu$ is the regularization parameter. Moreover, the measurement data $u^{s,\delta}$ and the exact data $u^{s}$ satisfies $\|u^{s,\delta}-u^{s}\|_2\leq\delta$.

There are many iterative algorithms to solve the variational problem (\ref{variational reg}).
In this paper, we apply the Levenberg-Marquardt  method \cite{Doicu2010, Hanke1997} to find the minimize of $J[q]$, which is essentially a variant of the Gauss-Newton iteration. Assuming that $q^*$ is an approximation of $q$, then the nonlinear mapping $\mathcal{F}$ in (\ref{variational reg}) can be replaced approximatively by its linearization around $q^*$, i.e.
\[
\mathcal{F}(q)\approx \mathcal{F}(q^*)+\mathcal{F}'(q^*)(q-q^*).
\]
The nonlinear inverse problem $\mathcal{F}(q)=u^{s,\delta}$ can then be converted to a linear inverse problem
\[
\mathcal{F}'(q^*)(q-q^*)=u^{s,\delta}-\mathcal{F}(q^*).
\]
Thus, minimizing (\ref{variational reg}) can easily be seen to minimizing
\begin{align*}
%\label{ equivalent variational reg}
J[q]:=\frac{1}{2}\parallel \mathcal{F}'(q^*)\delta q-(u^{s,\delta}-\mathcal{F}(q^*))\parallel_2^2+\frac{\mu}{2}\parallel \delta q \parallel_{L^2[0,2\pi]}^2,
\end{align*}
where, $\delta q=q-q^*$. We formulate the Levenberg-Marquardt iteration as Algorithm 1.
\begin{algorithm}[H]
\caption{Levenberg-Marquardt algorithm}
     \textbf{Input}: The regularization parameter $\mu$, the noise level $\delta$ and the maximum iterations $M$ \\
     \textbf{Output}: $q_{k+1}$.\\
     \textbf{Require}:\\
      ~1:~~\textbf{Initialize} the solution $q_0$;\\
      ~2:~~Settle the forward problem and obtain the additional data $u^{s,\delta}$;\\
      3:~~\textbf{While} $k< M$\\
      4:~~Compute the forward problem and set $F_k=u^{s,\delta}-\mathcal{F}(q_k);$\\
      5:~~Compute the Jacobian matrix $G=\mathcal{F}'(q_k)$;\\
      6:~~Compute $\delta q_k=(G^{T}G+\mu I)^{-1}(G^{T}F_k);$\\
      7:~~\textbf{Update} the solution: $q_{k+1}=q_{k}+\delta q_k;$\\
      8:~~$k\leftarrow k+1$;\\
      9:~~\textbf{Terminate} if $\parallel \delta q_k\parallel<eps $\\
      10:~~\textbf{end}
   %\label{algorithm-LM}
\end{algorithm}

We also decipher the inverse problem from a Bayesian perspective so that we can capture more statistical information about the solution. From the classical Bayesian theory, and noticing that the observation error $\xi$ is assumed to be an independent and identically distributed Gauss random vector with mean zero and the covariance matrix $B=\delta^2 I$ (here $I$ is the unit matrix), we can write the minimization functional as
\begin{align*}
J[q]&\propto \frac{1}{2\delta^2}\parallel \mathcal{F}(q)-u^\delta\parallel_2^2+\frac{\mu}{2\delta^2}\parallel q\parallel_{L^2[0,2\pi]}^2\\
&=\frac{1}{2}\parallel\mathcal{F}(q)-u^\delta \parallel_{B}^2+\frac{\mu}{2\delta^2}\parallel q\parallel_{L^2[0,2\pi]}^2\\
&=:J_B(q),
\end{align*}
where $\parallel \cdot\parallel_B$ is a covariance weighted norm given by $\parallel\cdot\parallel_{B}=\parallel\delta^{-1}I\cdot\parallel_{2}$, and the minimizer of $J_B$ defines the maximum a posteriori estimator
\begin{align}
\label{MAP}
q_{\mathrm{MAP}}=\arg \min_{q} J_B(q).
\end{align}

The Laplace approximation replaces the complicated posterior with a normal distribution located at the maximum a posteriori value $q_{\mathrm{MAP}}$. Its essence is a linearization around the $\mathrm{MAP}$ point $q_{\mathrm{MAP}}$  (cf. \cite{Schillings2020}). It consists of approximating the posterior measure (or distribution) by
$\omega\approx N(q_{\mathrm{MAP}},C_{\mathrm{MAP}})$, where
\begin{align}
\label{C_map}
C_{\mathrm{MAP}}=(J''_B)^{-1}=(\frac{\mu}{\delta^2}I+\frac{1}{\delta^2}G^{T}G)^{-1},
\end{align}
and $G$ is the Jacobian matrix of the forward operator $\mathcal{F}$ at the point $q$. Notice that the covariance formula ($\ref{C_map}$) only uses the first order derivatives of $\mathcal{F}$. The implementation of the Laplace approximation is presented in the following algorithm \cite{Iglesias2013}.

\begin{algorithm}[H]
\small
\caption{Laplace approximation (LA) for sampling.}
      1:~~Compute $q_{\mathrm{MAP}}$ from $(\ref{MAP})$ by using Algorithm 1, and $C_{\mathrm{MAP}}$ from $(\ref{C_map})$, respectively;\\
      ~2:~~Compute the \textcolor[rgb]{0,0,0}{Cholesky factor $L$} of $C_{\mathrm{MAP}}$, i.e., $C_{\mathrm{MAP}}=LL^{T}$;\\
      %\begin{eqnarray}
      %C_{\mathrm{MAP}}=LL^{T};
      %\end{eqnarray}
      ~3:~~For $j=\{1,...,N_e\}$, generate $q^{j}=q_{\mathrm{MAP}}+L^{T}z^{j}$,
      %\begin{eqnarray}\label{5.67}
      %q^{j}=q_{\mathrm{MAP}}+L^{T}z^{j},
      %\end{eqnarray}
      where $z^{j}\sim N(0,I)$.\\
   \label{algorithm-LA}
\end{algorithm}
In this algorithm, samples generated by algorithm \ref{algorithm-LA} are drawn from $N(q_{\mathrm{MAP}} ,C_\mathrm{MAP})$, and so the ensemble of $N_e$
realizations $\{q_j\}_{j=1}^{N_e}$ provides an approximation to $N(q_{\mathrm{MAP}} ,C_\mathrm{MAP})$ and hence the posterior.
Finally, we use the mean of the sample $\bar{q}=\frac{1}{N_e}\sum_{j=1}^{N_e}q_j$ as an approximation of $q_{\mathrm{MAP}}$, where the
convergence of $\bar{q}$ follows from the strong law of large numbers. From the classical Gaussian statistic theory, we find that $\bar{q}$ is consistent and the best unbiased estimate of $q_{\mathrm{MAP}}$.

\section{ Numerical results and discussions }
In this section, we present several numerical examples to illustrate the salient and promising features of the proposed reconstruction scheme.

In all of our numerical examples, the measurement boundary curve $\partial\Omega$ is given by the circle of radius 3 and centred at the origin, that is $\partial \Omega=\{3(\cos t, \sin t ), 0\leq t\leq 2\pi\}$. Set the incident field $H(x)=x_1$. We solve the direct problem through the $\mathrm{Nystr\ddot{o}m}$ method \cite{Kress1999}, which is discretized with $n=80$ grid points. Furthermore, in an effort to avoid committing an inverse crime, the number of collocation points for obtaining the synthetic data was chosen to be different from the number of collocation points within the inverse solver. In addition, we approximate the radial function $q(t)$ for unknown interior boundary curve $\partial D$ by the trigonometric series
\[
q(t)\approx \sum_{k=0}^m a_k\cos kt +\sum_{k=1}^m b_k\sin kt, \ \ \ 0\leq t\leq 2\pi,
\]
where $m\in\mathbb{N}$ and the vector $q=(a_0, ..., a_m, b_1,...,b_m)\in \mathbb{R}^{2m+1}$.

 In the iterative process, we choose a circle as the initial guess, which contains the inclusion $D$. \textcolor[rgb]{0,0,0}{A finite difference method is used to calculate the Jacobian matrix $G$}, and the maximum number of iteration steps is 100.  The number of samples $N_e$ is 10000, and we use the following stopping rule
\[
E_k=\parallel q_k-q_{k-1}\parallel_{L^2}\leq 10^{-5}.
\]
The noisy measured data are generated by
\[
u^{s,\delta}=u^s(x)+\delta\xi, \ \ x\in\partial\Omega,
\]
where $u^s(x)$ is the exact data, $\delta$ indicates the noise level, and $\xi$ is the Gaussian random vector with a zero mean and unit standard deviation.

The accuracy of the approximate solution $\tilde{q}_{N_e}$ by the LA algorithm is characterised by comparing to the exact solution $q(x)$  via the relative error
\[
e_\gamma=\frac{\parallel \tilde{q}_{N_e}-q(x) \parallel_{L^2(\partial D)}}{\parallel q(x)\parallel_{L^2(\partial D)}}.
\]

In particular, we specifically take the permittivity $\varepsilon_m=24.8 F/m$ (such as alcohol) in the first example, and in all of the examples we set $\varepsilon_m=\varepsilon_0$. {Moreover, except for Example~\ref{exm-disk}, the spectrum of $\mathcal{K}_D^*$ needs to calculated numerically, from which one can then compute the corresponding plasmon resonance frequency. We use the $\mathrm{Nystr\ddot{o}m}$ method for spectral calculations. However, we would like to emphasise that in practice, the plasmon resonance frequency can be measured by special apparatus, when one can observe significant enhancements in the electric field due to the resonance.}

\begin{exm}\label{exm-disk}
In this example, we consider the reconstruction of a circle object with
\[
q(t)=0.5, \ \ \ 0\leq t\leq2\pi.
\]
In the iteration algorithm, we choose a circle of radius 0.6 as the initial guess.
\end{exm}

First, we investigate the influence of  different $\lambda$ on the reconstruction effect. The numerical results for Example \ref{exm-disk} with  $\lambda$  taking different values are shown in Figure \ref{disk} (a). It is well known that the eigenvalues of $\mathcal{K}_D^*$ are $\{0,\frac{1}{2}\}$ for $D$ being a disk. If setting $\lambda_1=0$, from Section 2.2, we can see that the Fr\"{o}hlich condition is satisfied based on the Drude's model for metal nanorod and the plasmon resonance occurs {(though our theoretical analysis does not contain the essential spectrum case).}
%According to Theorem \ref{spectral} and Corollary \ref{cor3.6}, it finds that the plasmon resonance will enhance the sensitivity of the shape construction greatly.
In particular,  we take the resonance frequency $\omega_1=3.63\cdot10^{14} Hz$, and according to the experimental data of gold nanorod  in \cite{Ordal1983} (the data coincide with the Drude model), we can get $\mathfrak{Re}(\varepsilon_D(\omega_1))=-24.8$, and the imaginary part $\mathfrak{Im}(\varepsilon_D(\omega_1))=0.797$. Furthermore, by (\ref{lambda}), it implies that $\lambda\approx0-8\cdot10^{-3}i$.

\begin{figure}
\centering
\subfigure[]{
    \label{fig:subfig:a}
  \includegraphics[width=2.5in, height=2.2in]{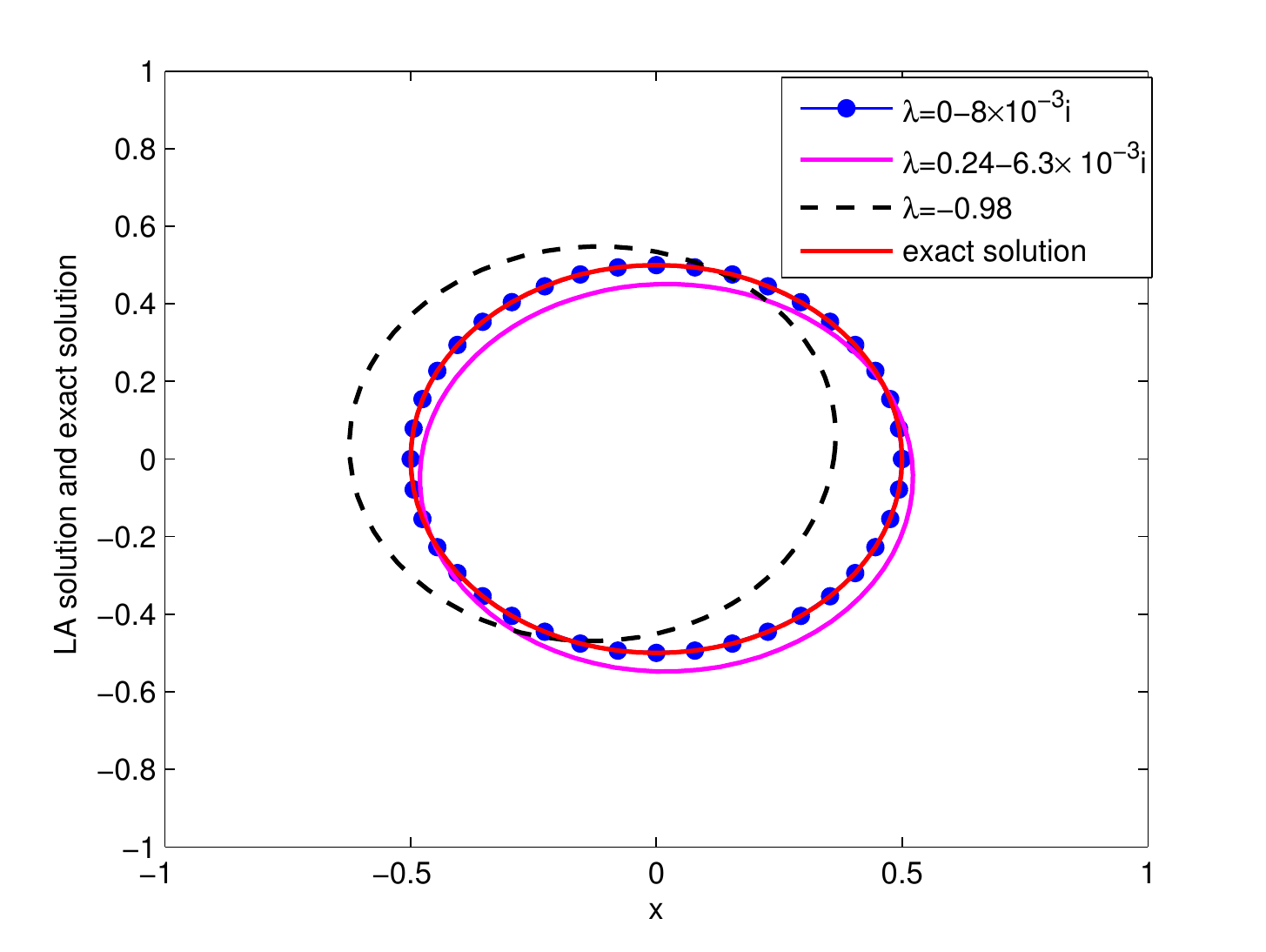}}
  \subfigure[]{
    \label{fig:subfig:b}
  \includegraphics[width=2.5in, height=2.2in]{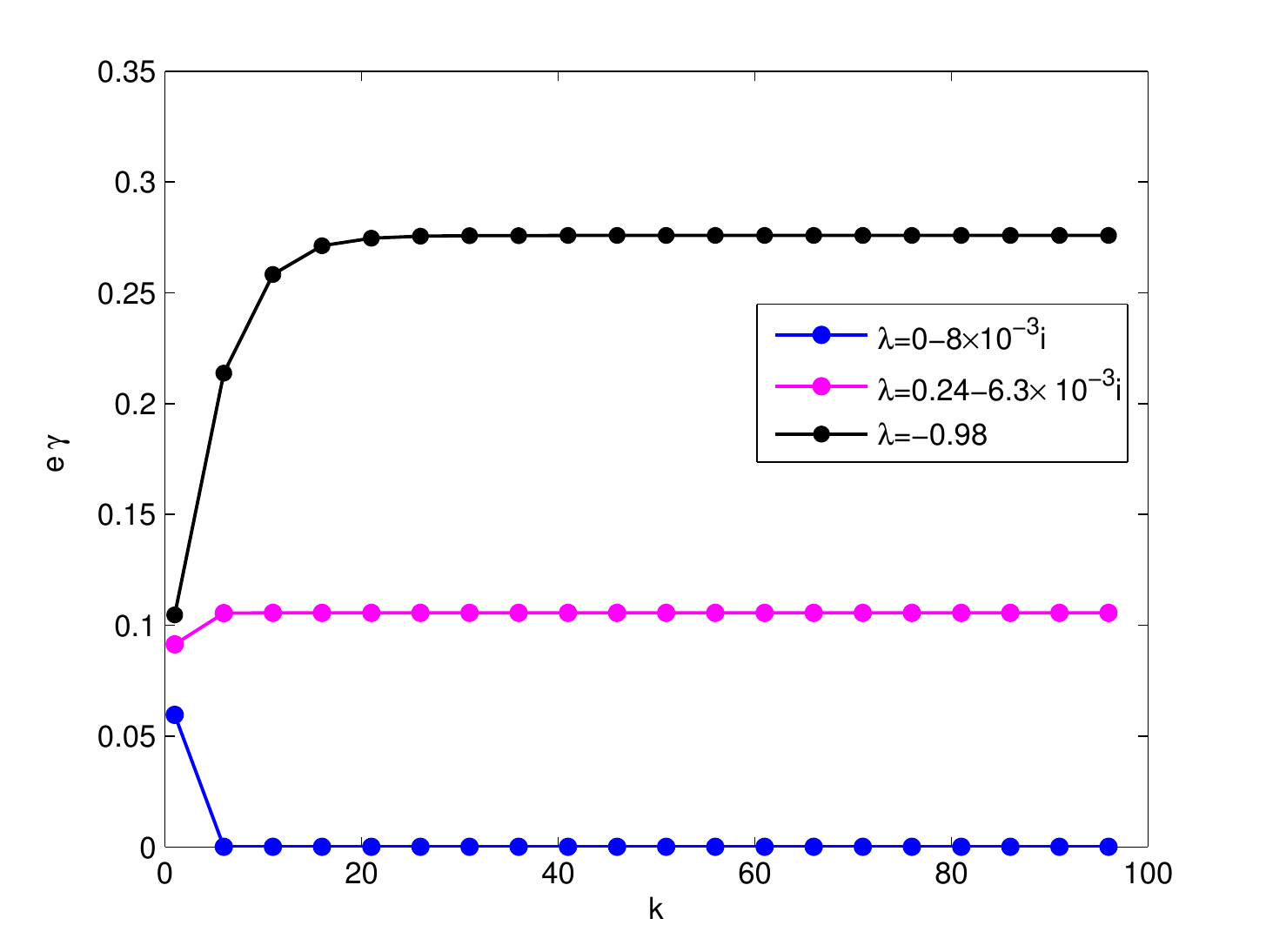}}
  \caption{Reconstruction the shape  for Example \ref{exm-disk} with 1\% noise data  and the regularization parameter $\mu = 0.01$.  (a) Reconstructions with different $\lambda$ as well as the exact solution, (b) the relative error $e_\gamma$ versus the iterations step $k$, associated to different values of $\lambda$. }
  \label{disk}
\end{figure}
Next, we take $\omega_2=2.30\cdot10^{14} Hz$ and based on the experimental data in \cite{Ordal1983}, we can calculate $\lambda=0.244-6.3\cdot10^{-3}i$. It is obvious that when taking $\omega_2$, the real part of $\lambda(\omega_2)$ satisfies $|\mathfrak{Re}(\lambda(\omega_2))<\frac{1}{2}|$, which is only an approximate resonance frequency. As seen in Figure \ref{disk} (a),  a perfect reconstruction result is obtained when taking the resonance frequency $\omega_1$. When $\lambda=0.244-6.3\cdot10^{-3}i$, although the reconstructed result is good, it is still poor compared to the case with the resonant frequency $\omega_1$. The case with $\lambda =-0.98$  indicates a normal material, i.e. the permittivity $\varepsilon_D= 8$ is positive constant, and the reconstruction effect is poorer.
The  plot in Figure \ref{disk} (b) illustrates the relative error $e_\gamma$ versus the iteration number with different $\lambda$. It can be seen that  the relative error  can quickly reach the convergence  and remain small when $\omega_1$ is the plasmon resonance frequency. When $\lambda=0.244-6.3\cdot10^{-3}i$, the relative error also converges quickly, but is larger than that at the plasmon resonance frequency. In particular, the relative error increases rapidly even up to 27\%, and the inversion is rather imprecise with $\lambda=-0.98$.
\begin{figure}
\centering
\subfigure[]{
    \label{fig:subfig:a}
  \includegraphics[width=2.5in, height=2.2in]{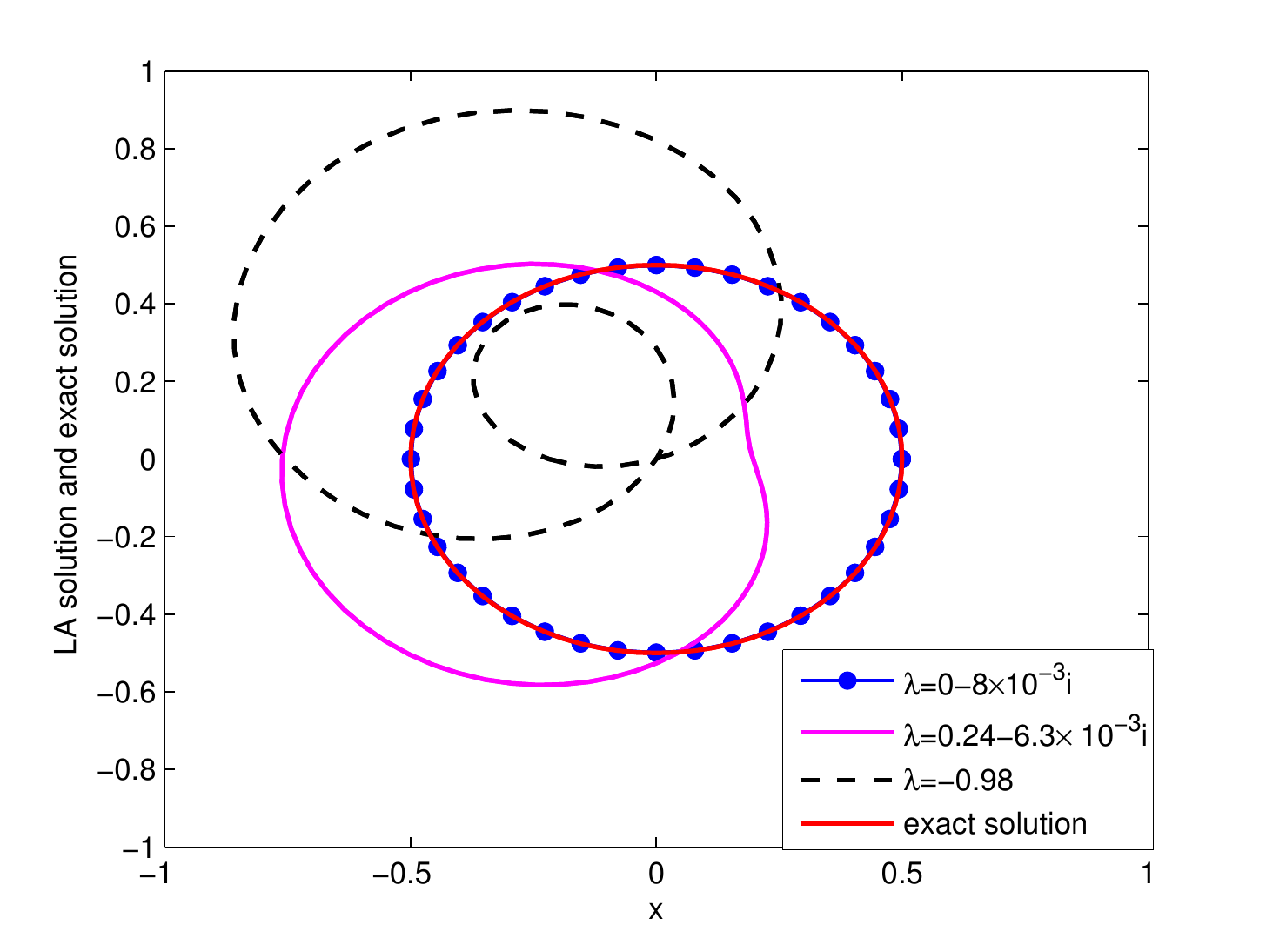}}
  \subfigure[]{
    \label{fig:subfig:b}
  \includegraphics[width=2.5in, height=2.2in]{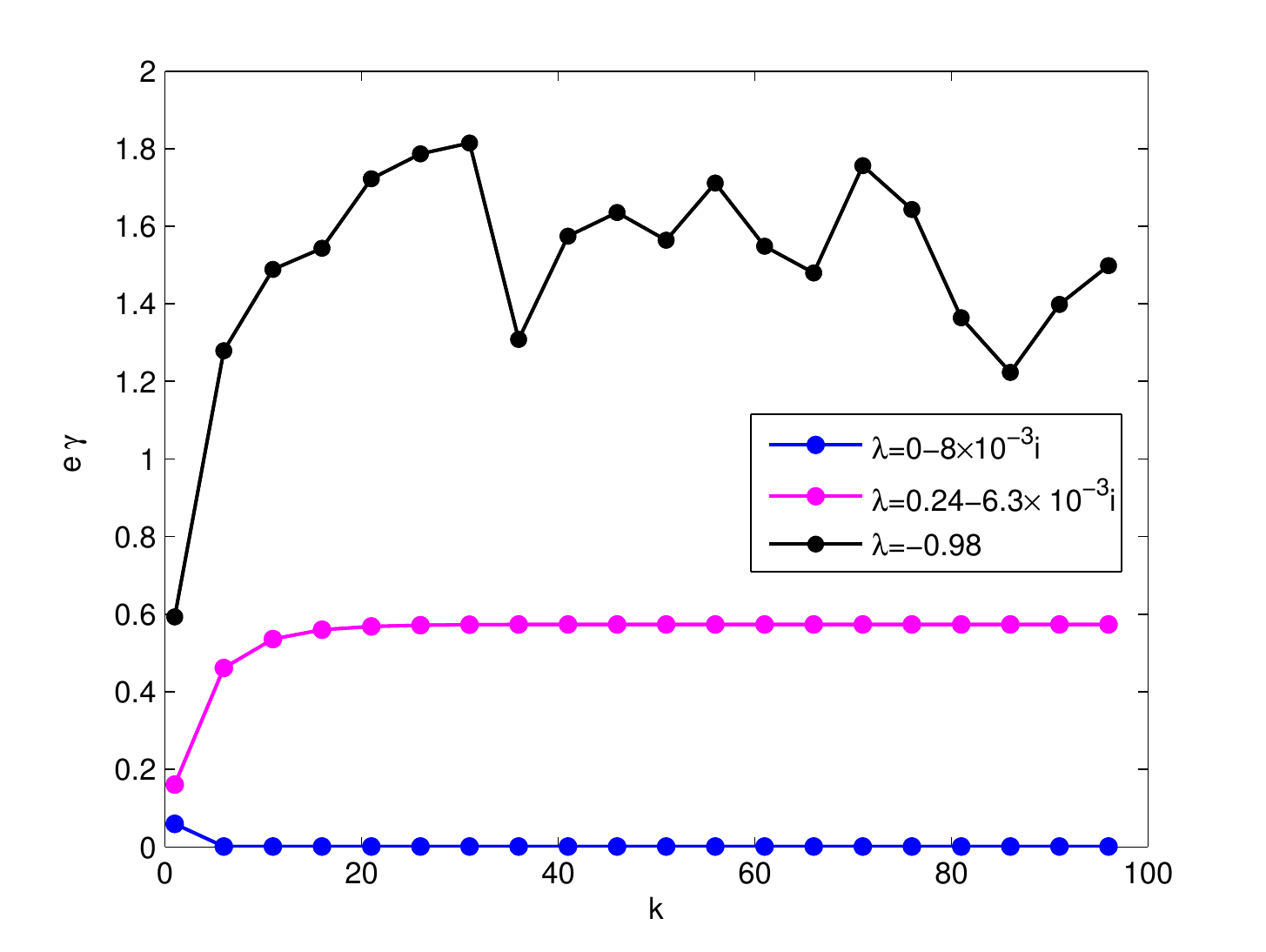}}
  \caption{Reconstruction of the shape for Example \ref{exm-disk} with 5\% noise data and the regularization parameter $\mu=0.05$. (a) Reconstruction under different $\lambda$  as well as the exact solution, (b) the accuracy error $e_\gamma$ of the number of the iterations $k$, for different values of $\lambda$. }
  \label{disk5}
\end{figure}

The numerical results for Example \ref{exm-disk}, which has $5\%$ noise in the data, are shown in Figure \ref{disk5}(a) for different $\lambda$. Compared to the $1\%$ noise in the data shown in Figure \ref{disk}(a), it is clear that as the noise level increases, the reconstruction effect becomes worse. It is worth noting that when $\omega$ is a plasmon resonance frequency, the reconstruction is very good even at larger error levels. For $\lambda=-0.98$ (general materials) the inversion result deviates from the exact solution quickly as the error level increases. We also list the relative errors for different noise levels in Table \ref{tab1}.

\begin{table}
\centering
\caption{ Numerical results of Example \ref{exm-disk} for $e_\gamma$ associated with different $\lambda$ and noise level $\delta$ }
\begin{tabular}{c|c|c}
\Xhline{1pt}
 $\lambda$ &$ \delta=0.01$ &$\delta=0.05$\\
  \hline
  $0-8\cdot10^{-3}i$  &$7.2\cdot10^{-5}$ &$9.4\cdot10^{-4}$  \\
 \hline
 $0.24-6.3\cdot10^{-3}i$ &$0.106$&$0.574$\\
  \hline
 $-0.98$     &$0.276$&$1.546$ \\
 \hline
\Xhline{1pt}
\end{tabular}
\label{tab1}
\end{table}
\begin{exm}\label{exm2}
 In this example, we consider the reconstruction of a bean contour with a radial function,
\[
q(t)=\frac{4/5+18/25\cos t+3/25\sin 2t}{1+7/10\cos t},\ \ \ 0\leq t\leq2\pi.
\]
We choose a circle of radius 0.72 as the initial guess.
\end{exm}

\begin{figure}
\centering
\subfigure[]{
    \label{fig:subfig:a}
  \includegraphics[width=2.5in, height=2.2in]{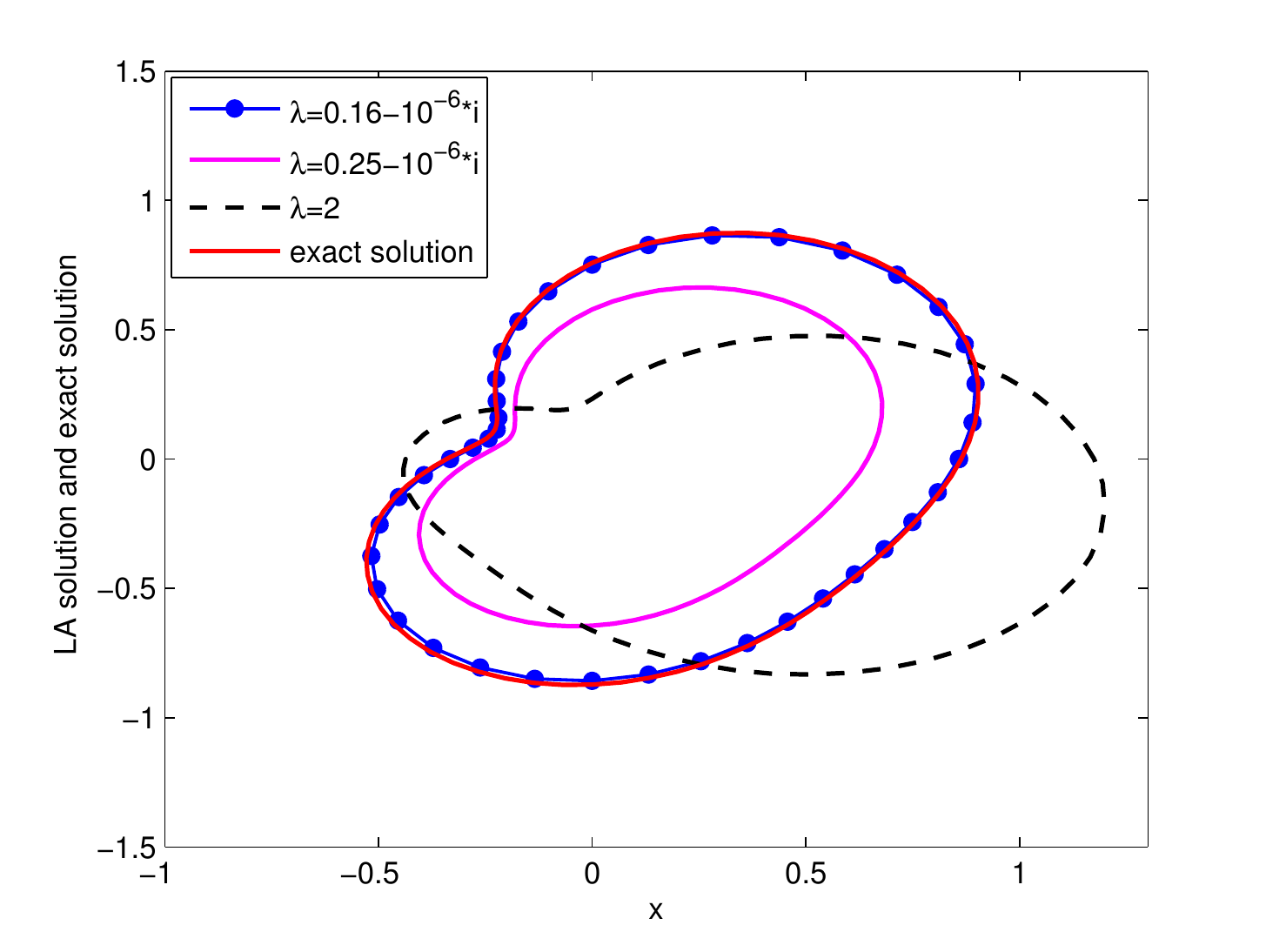}}
  \subfigure[]{
    \label{fig:subfig:b}
  \includegraphics[width=2.5in, height=2.2in]{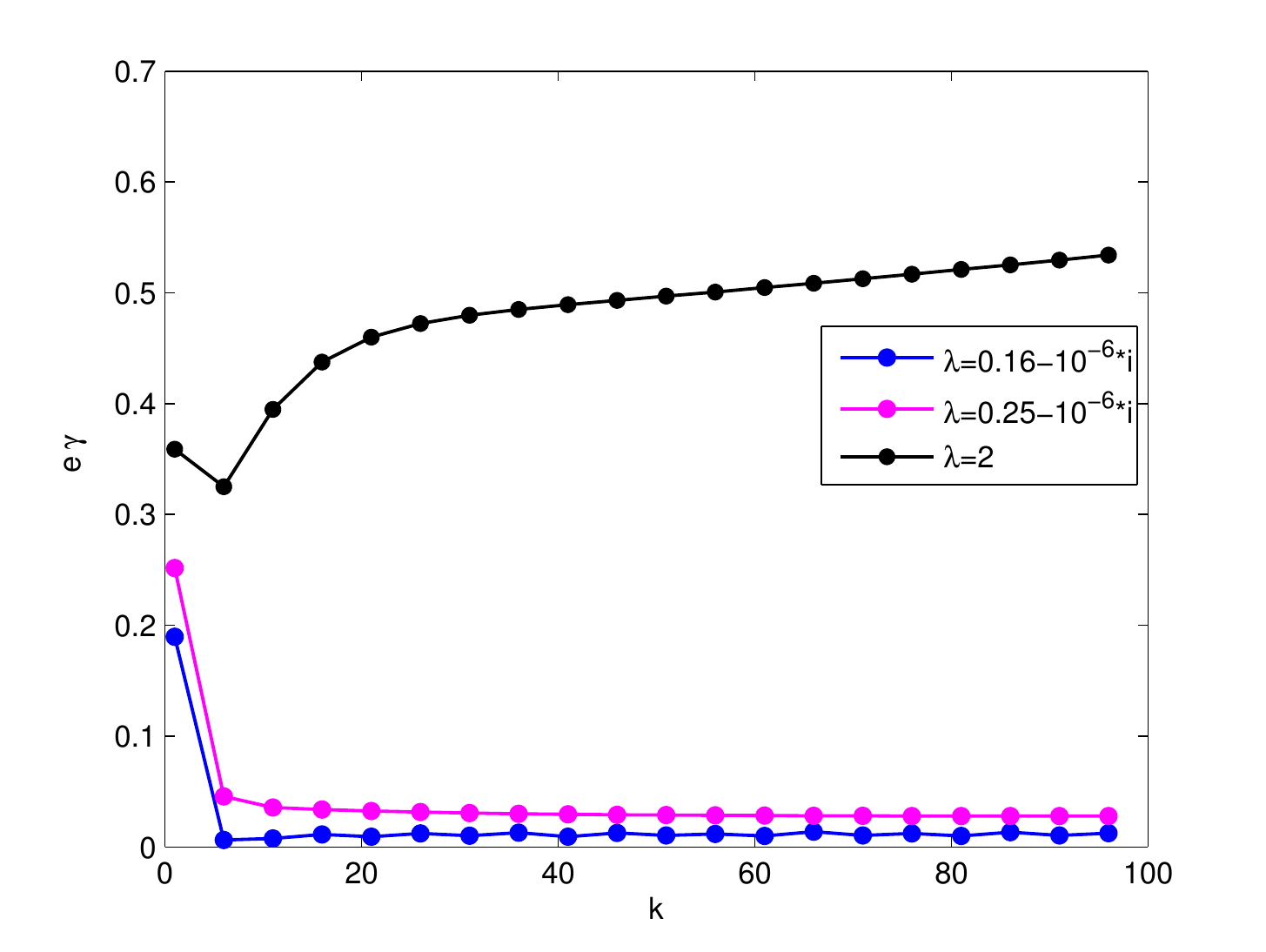}}
  \caption{Reconstruction of the shape in Example \ref{exm2} with 1\% noise data and the regularization parameter $\mu=0.1$. (a) Reconstruction with different $\lambda$ as well as the exact solution, (b) the relative error $e_\gamma$ versus the iteration step $k$ associated with different values of $\lambda$. }
  \label{examp2-001}
\end{figure}
The numerical results for Example \ref{exm2} with $\lambda$ taking different values are shown in Figure \ref{examp2-001} (a) and the relative error $e_\gamma$ versus the iteration number with different $\lambda$ in Figure \ref{examp2-001} (b). The real part of $\lambda=0.16-10^{-6}i$ is the eigenvalue of $\mathcal{K}^*_{D}$ in Example \ref{exm2}. It represents very accurate approximations of the LA solution to the exact solution, and in the iteration algorithm, the relative error quickly reaches a stable state and remains small. In addition, when $\lambda=2$, the reconstruction is bad. The results show that the relative error becomes more oscillating as the noise level $\delta$ increases; see Figure \ref{examp2-005} and Table \ref{tab2} for $e_\gamma$ with different $\lambda$ and $\delta$.
\begin{figure}
\centering
\subfigure[]{
    \label{fig:subfig:a}
  \includegraphics[width=2.5in, height=2.2in]{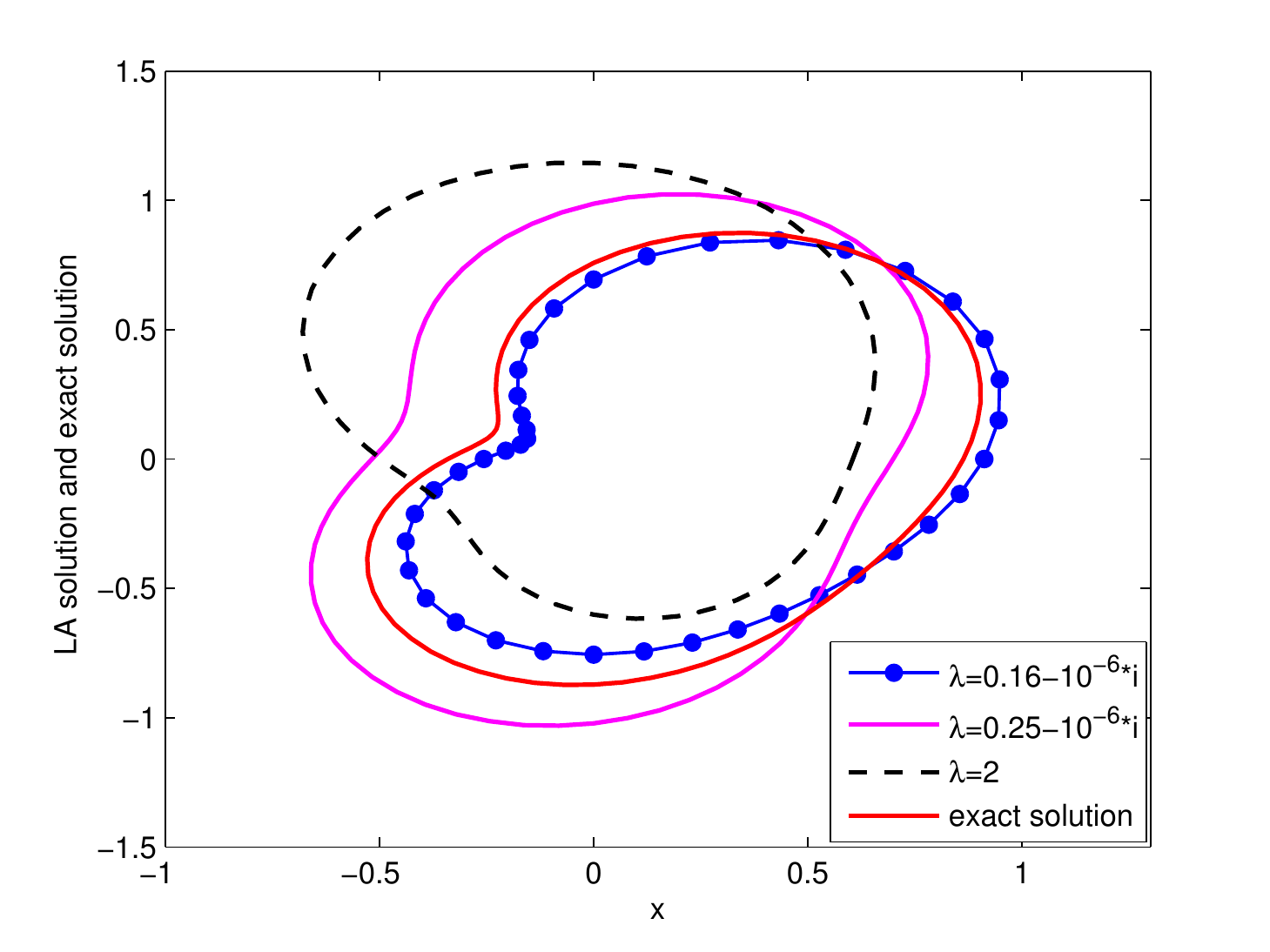}}
  \subfigure[]{
    \label{fig:subfig:b}
  \includegraphics[width=2.5in, height=2.2in]{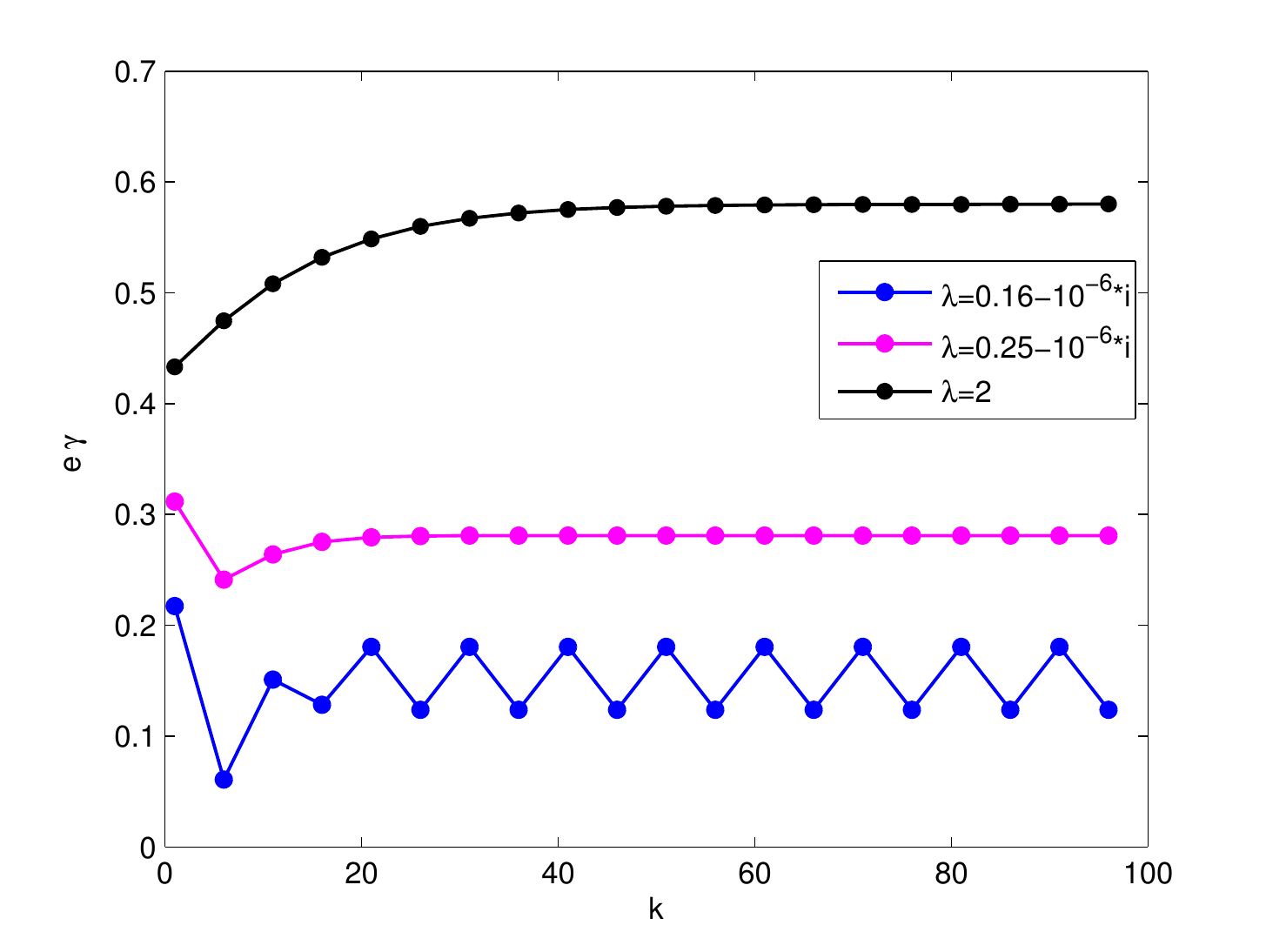}}
  \caption{Reconstruction of the shape in Example \ref{exm2} with 5\% noise data and the regularization parameter $\mu=0.1$. (a) Reconstruction with different $\lambda$ as well as the exact solution, (b) the relative error $e_\gamma$ versus iteration steps $k$ associated with different values of $\lambda$.}
  \label{examp2-005}
\end{figure}
\begin{table}
\small
\centering
\caption{ Numerical results of Example \ref{exm2} for $e_\gamma$ associated with different $\lambda$ and $\delta$ }
\begin{tabular}{c|c|c}
\Xhline{1pt}
 $\lambda $ &$ \delta=0.01$ &$\delta=0.05$\\
  \hline
  $0.16-10^{-6}i$  &$0.0081$ &$0.1237$  \\
 \hline
 $0.25-10^{-6}i$ &$0.0278$&$0.2810$\\
  \hline
 $2$     &$0.5304$&$0.5804$ \\
 \hline
\Xhline{1pt}
\end{tabular}
\label{tab2}
\end{table}

\begin{exm}\label{exm3}
In this example, we consider that the inclusion is a peanut, and the polar radius of the peanut is parameterized by
\[
q(t)=\sqrt{\cos^2 t+0.26\sin^2 (t+0.5)},\ \ \ 0\leq t\leq2\pi.
\]
In the iteration, we choose a circle of radius 0.78 as the initial guess.
\end{exm}

From \cite{Banks2009}, it is known that the singular value decomposition of the sensitivity matrix plays a key role in the uncertainty estimation. Let the singular value decomposition (SVD) of the sensitivity matrix (Jacobian matrix) $G$ of the forward operator at the true solution $q_{true}$ be denoted as
\begin{eqnarray*}
\label{Qresidual-eq}
\begin{split}
G(q_{true})=U \bigg{[}\begin{array}{c}
\mathrm\Lambda\\
0
\end{array} \bigg{]} V^T,
\end{split}
\end{eqnarray*}
where $U$ is an $n\times n$ orthogonal matrix, i.e. $U^TU=UU^T=I_n$, with $U_1$ containing the first $2m+1$ columns of $U$ and $U_2$ containing the last $n-(2m+1)$ columns, $U=[U_1\ U_2]$. The matrix $V$ is an $(2m+1)\times (2m+1)$ orthogonal matrix, i.e. $V^TV=VV^T=I_p$, and $v_i$ and $u_i$ denote the $i$th columns of $V$ and $U$, respectively. The diagonal matrix $\Lambda=diag(s_1,...,s_{2m+1})$ with strictly positive decreasing singular values $s_i$, i.e. $s_1\geq s_2\geq ...\geq s_{2m+1}\geq0$. Then the estimator $q$ has the following form \cite{Banks2009}:
\begin{align}
\label{SVD}
q=q_{true}+V\Lambda^{-1}U_1^T\tilde{\xi}=q_{true}+\sum_{i=1}^{2m+1}\frac{1}{s_i}v_iu_i^T\tilde{\xi}.
\end{align}
From $(\ref{SVD})$, it can be seen that the instability of the inverse problem is caused by the small singular values. In Example \ref{exm3}, the singular values of the sensitivity matrix are calculated at different values of $\lambda$.
In Figure \ref{examp3-001}(b), when $\lambda(\omega)=0.19-10^{-6}i$ ($\omega$ is a plasmon resonance frequency for Example \ref{exm3}), we can see that all the singular values of $G$ are lager than $\lambda=0.25-10^{-6}i$ and $\lambda=2$. Consequently, the numerical solution computed by using the plasmon resonance is in an excellent agreement with the exact solution (see Figure \ref{examp3-001}(a)). Hence, the plasmon resonance of the metal nanoparticles can correct the singular value of the sensitivity matrix and overcome the numerical instability. This numerical result is in agreement with our analysis in Section 3.2 that the sensitivity of the {far-field} data to the shape of the underlying domain can be enhanced at the resonant frequencies, thus reducing the ill-posedness of the inverse problem.

\begin{figure}
\centering
\subfigure[]{
    \label{fig:subfig:a}
  \includegraphics[width=2.5in, height=2.2in]{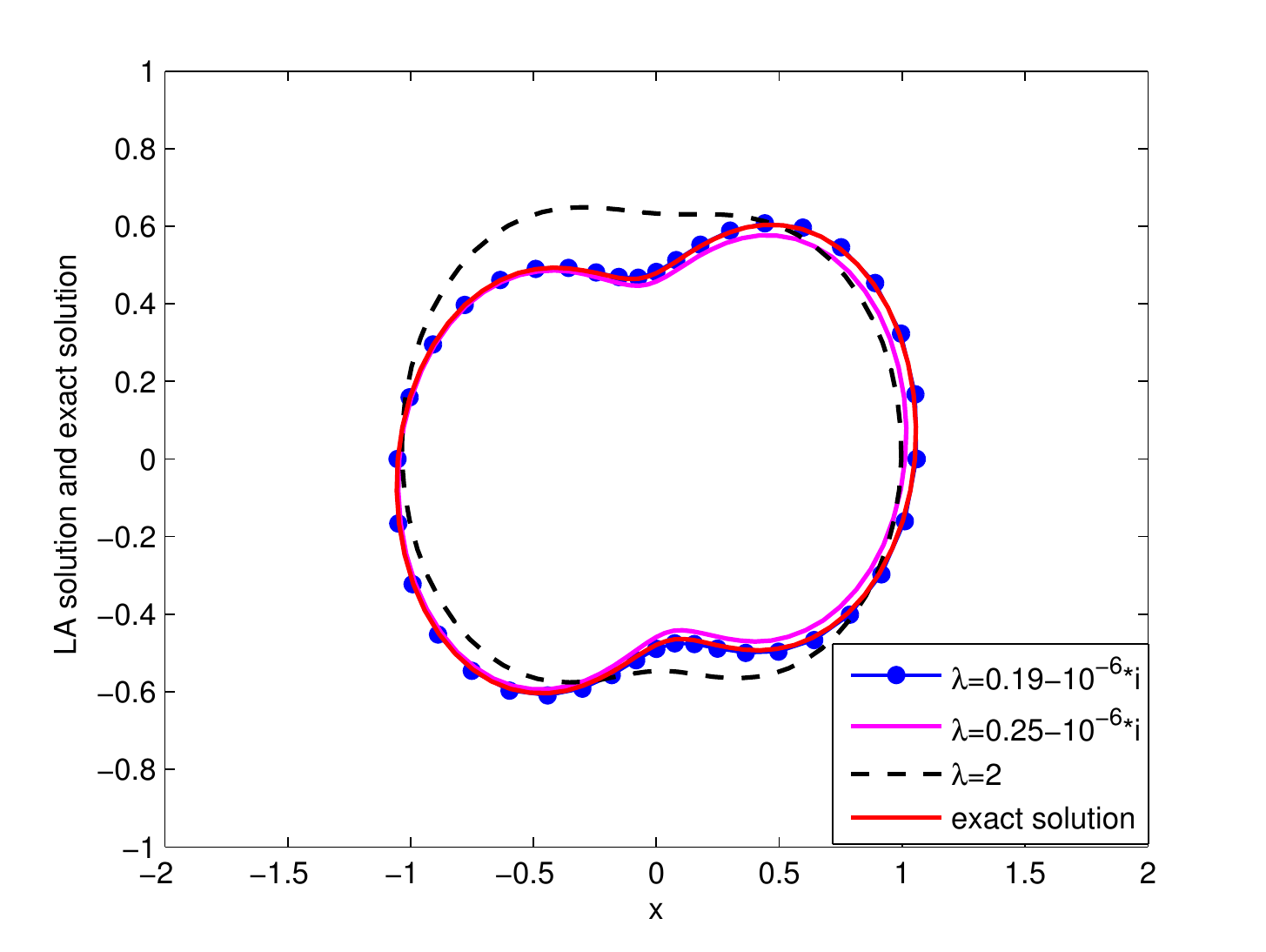}}
  \subfigure[]{
    \label{fig:subfig:b}
    \includegraphics[width=2.5in, height=2.2in]{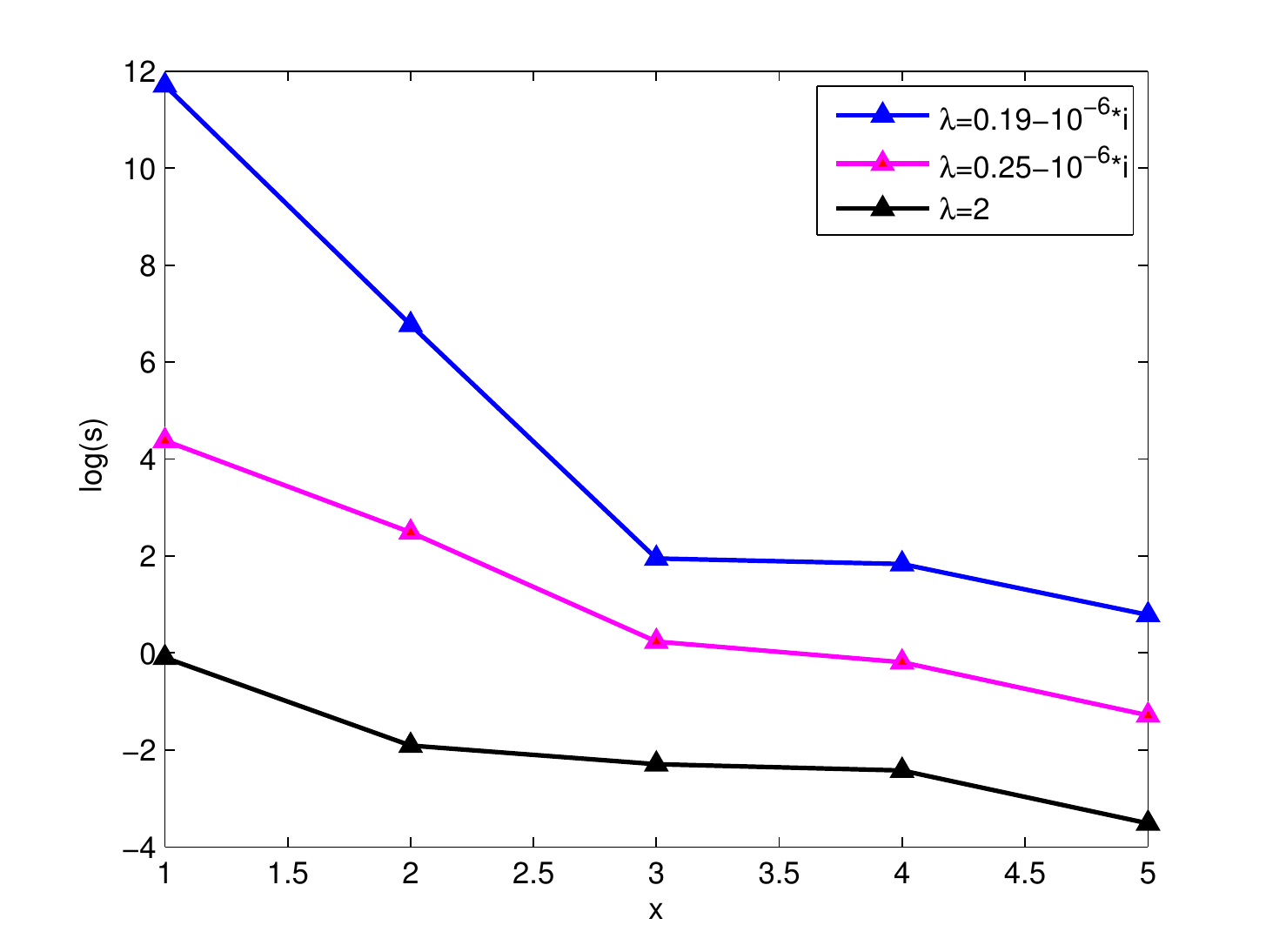}}
  \caption{(a) Reconstruction of the shape in Example \ref{exm3} with 1\% noise data and different $\lambda$ as well as the exact solution, under the same regularization parameter $\mu=0.05$, (b) the distribution of the singular values associated to different values of $\lambda$.  }
  \label{examp3-001}
\end{figure}
%We can also use the parameter selection score function $\beta(q)$ (see \cite{H.T. Banks2009} ) to calculate the corresponding standard deviation (uncertainty quantification) of $q$,  we define
%parameter selection score function such that
%\begin{align}
%\beta(q)=|\upsilon(q)|,
%\end{align}
%here,
%\begin{align}
%\upsilon(q)=\frac{\sqrt{C_{ii}}}{q_i},\ \  C(q)=\delta^2[G^{T}G]^{-1}\in \mathbb{R}^{2m+1\times 2m+1},
%\end{align}
%the components of the $\beta(q)$  are the ratios of each standard error for
%a parameter to the corresponding nominal parameter value, and noting that $\beta(q)$ near zero indicates lower uncertainty in the estimation. In Table \ref{tab3}, we can conclude that for the same noise level as the distance of $dist(\lambda,\sigma(\mathcal{K}^*_D))$ decreases, the scores $\beta(q)$ becomes progressively smaller, which indicating the sensitivity to random error can be reduced.
\begin{figure}[h]
\centering
\subfigure[$\lambda=0.19-10^{-6}i$]{
    \label{fig:subfig:a}
  \includegraphics[width=2.5in, height=2.2in]{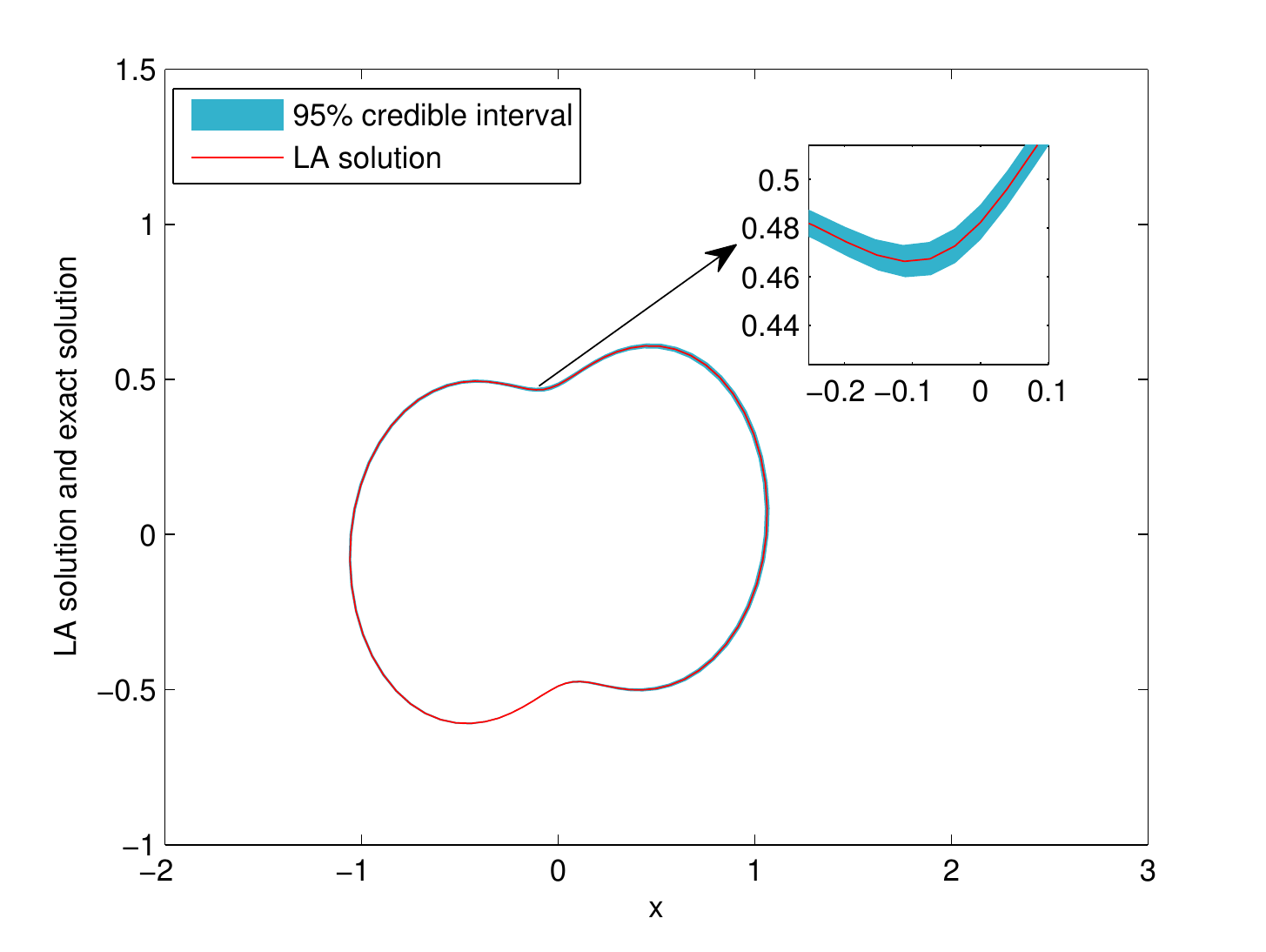}}
  \subfigure[$\lambda=2$]{
    \label{fig:subfig:b}
  \includegraphics[width=2.5in, height=2.2in]{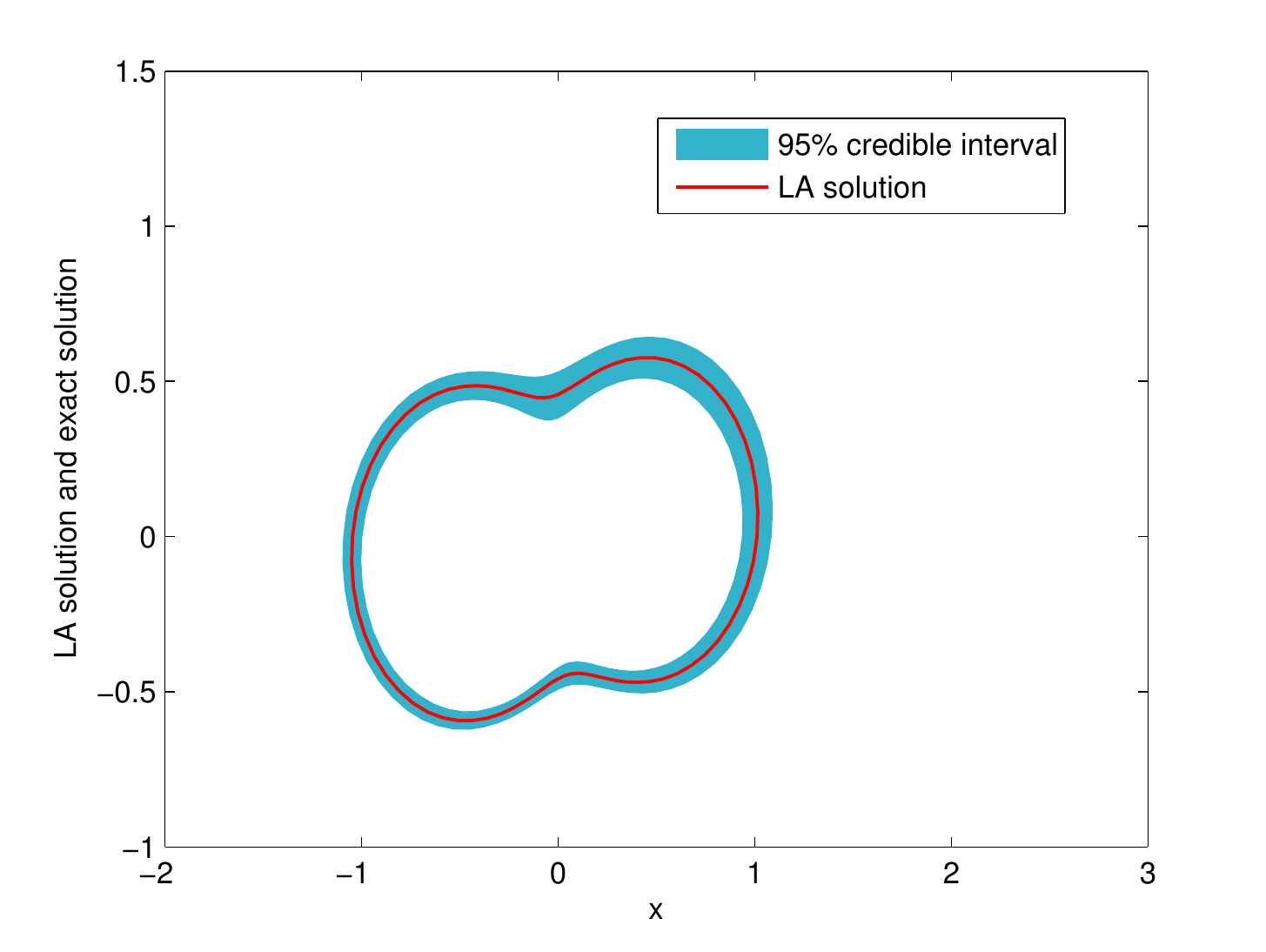}}
  \caption{The numerical results for Example \ref{exm3} with 1\% noise in the data and 95\% confidence interval for different $\lambda$.}
  \label{con-interval}
\end{figure}

Next, we study the variations of the confidence intervals with different $\mathrm{dist}(\lambda,\sigma(\mathcal{K}^*_D))$. The confidence interval can quantify the uncertain information of the solution. The numerical results for Example \ref{exm3} with different $\lambda$ are shown in Figure \ref{con-interval}, where the blue region represents the corresponding 95\% confidence region. The comparison with $\lambda=2$ indicates that the confidence region shrinks at $\lambda=0.19-10^{-6}i$ (the real part of $\lambda=0.19-10^{-6}i$ is the eigenvalue of $\mathcal{K}^*_D$ in Example \ref{exm3}). It can be clearly obtained that as $\mathrm{dist}(\lambda,\sigma(\mathcal{K}^*_D))$ shrinks, the accuracy of the inversion can be improved and the sensitivity to the random error can be reduced.

\begin{exm}\label{exm4}
In this example, we consider that the inclusion is pear, and the polar radius of the pear is parameterized by,
\[
q(t)=18/25+3/20\cos(3t),\ \ \ 0\leq t\leq2\pi.
\]
In the iteration, we choose a circle of radius 0.73 as the initial guess.
\end{exm}

In Example \ref{exm4}, we consider reconstructing a more challenging pear-shaped inclusion and we can actually reach similar conclusions to previous three Examples. At the same noise level, $\tilde{q}_{N_e}$ coincides well with the exact solution when $\lambda=0.14-10^{-6}i$ (minimum distance of $\mathrm{dist}(\lambda,\sigma(\mathcal{K}^*_D))$, and a steady, fast convergence of the relative error $e_\gamma$ in the iteration algorithm is shown in Figure \ref{examp4-001}. Moreover, the relative error becomes more oscillating as the noise level $\delta$ increases; see Table \ref{tab4} for $e_\gamma$ associated with different $\lambda$ and $\delta$.
\begin{figure}
\centering
\subfigure[]{
    \label{fig:subfig:a}
  \includegraphics[width=2.5in, height=2.2in]{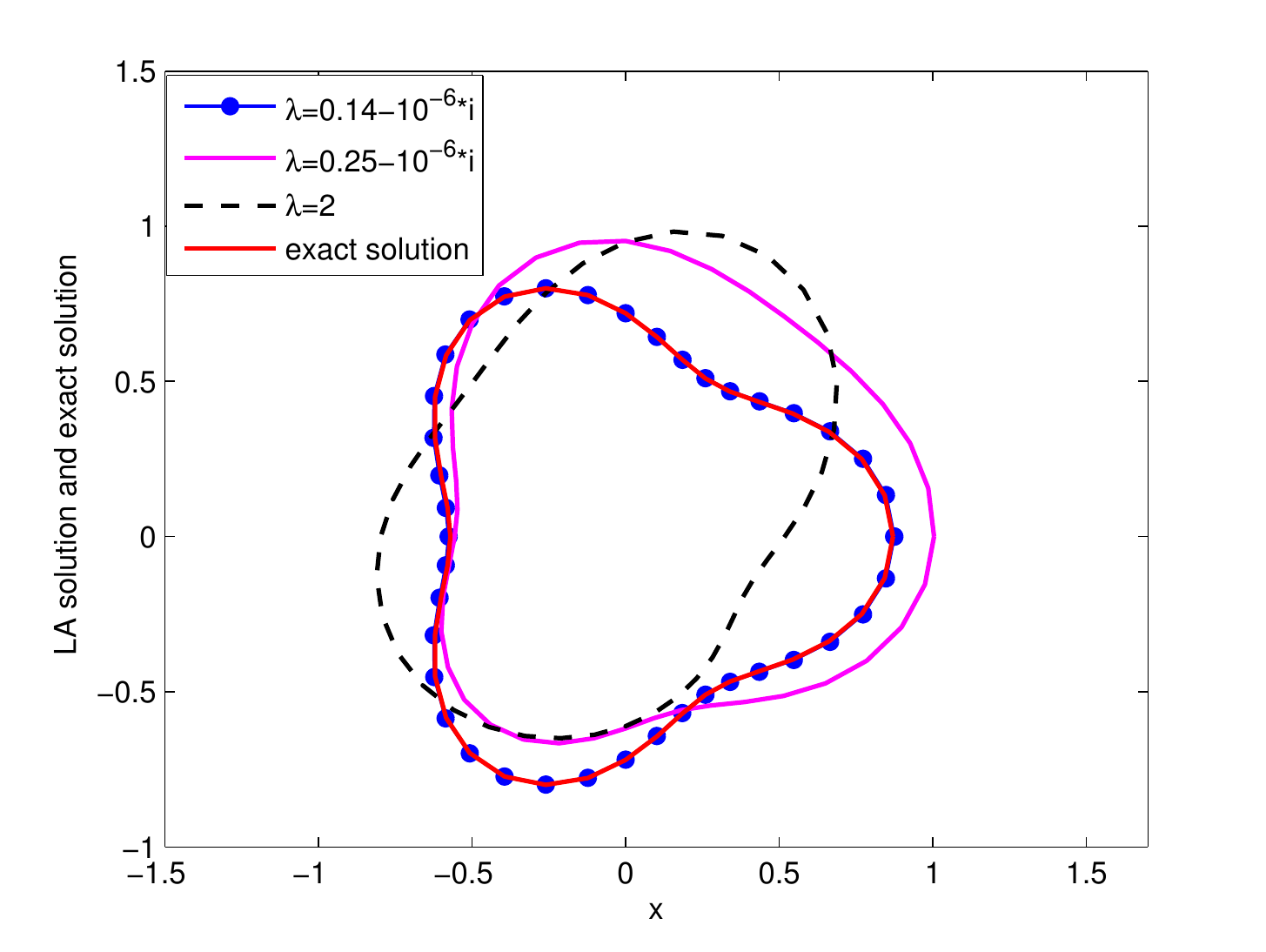}}
  \subfigure[]{
    \label{fig:subfig:b}
  \includegraphics[width=2.5in, height=2.2in]{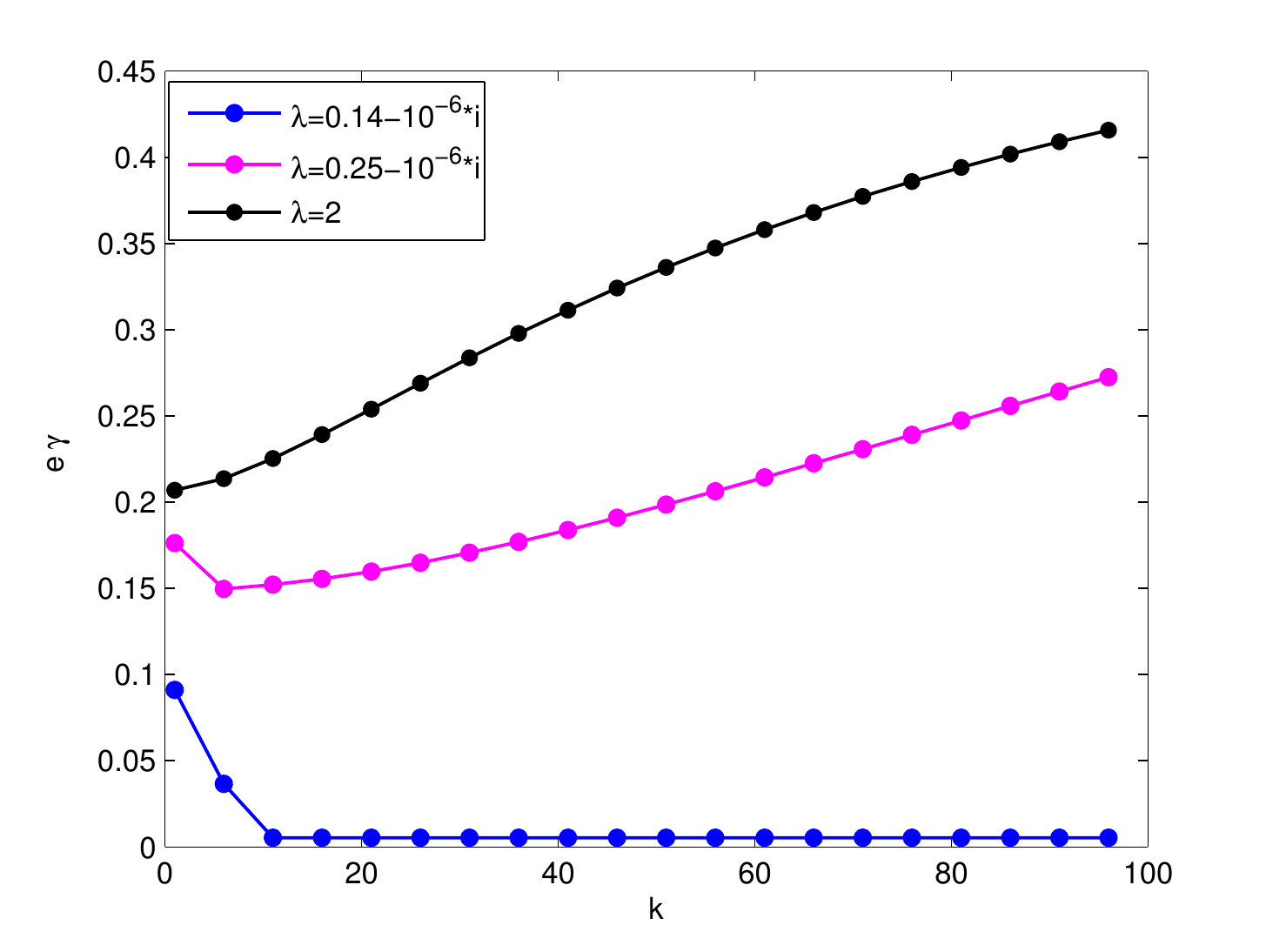}}
  \caption{Reconstruction the shape  for Example \ref{exm4} with 1\% noise data and regularization parameter $\mu=0.5$. (a) Reconstruction with different $\lambda$ and exact solution, (b) the relative error $e_\gamma$ versus iterations  steps $k$, for various values of $\lambda$. }
  \label{examp4-001}
\end{figure}

\begin{table}
\centering
\caption{ Numerical results of Example \ref{exm4} for $e_\gamma$ with various $\lambda$ and noise level $\delta$ }
\begin{tabular}{c|c|c}
\Xhline{1pt}
 $\lambda $ &$ \delta=0.01$ &$\delta=0.05$\\
  \hline
  $0.14-10^{-6}i$  &$0.0054$ &$0.0625$  \\
 \hline
 $0.25-10^{-6}i$ &$0.279$&$0.4850$\\
  \hline
 $2$     &$0.4206$&$1.009$ \\
 \hline
\Xhline{1pt}
\end{tabular}
\label{tab4}
\end{table}

%\begin{table}[H]
%\small
%\centering
%\caption{ Numerical results of Example 2 for Err with various values of $\lambda$  }
%\begin{tabular}{c|c|c|c}
%\Xhline{1pt}
% $\lambda $ &$cond$ &$s_{min}$ & $error$\\
%  \hline
%  $0.16+10^{-6}i$ &$1.4366\times 10^5$  & $3.6257$  & $0.0054$ \\
% \hline
% $0.25+10^{-6}i$ &$1.34\times10^3$ & $0.0105$ & $0.2790$ \\
%  \hline
% $2$    &$160.82$  & $0.0048$ & $0.4209$ \\
% \hline
%\Xhline{1pt}
%\end{tabular}
%\label{tab1-basis}
%\end{table}
%\begin{figure}[H]
%\centering
%\subfigure[]{
%    \label{fig:subfig:a}
%  \includegraphics[width=2.5in, height=2.2in]{LQ005}}
%  \subfigure[]{
%    \label{fig:subfig:b}
%  \includegraphics[width=2.5in, height=2.2in]{LErr005}}
%  \caption{The numerical results for Example 1. (a) is $\lambda=0.2+0.0001i$ and err=0.0057, (b) is $\lambda=2+0.0001i$ and err=0.0043. $\delta=0.05$ }
%  \label{ MCMC }
%\end{figure}
\begin{figure}
  \centering
  % Requires \usepackage{graphicx}
  \includegraphics[width=3.3in, height=2.5in]{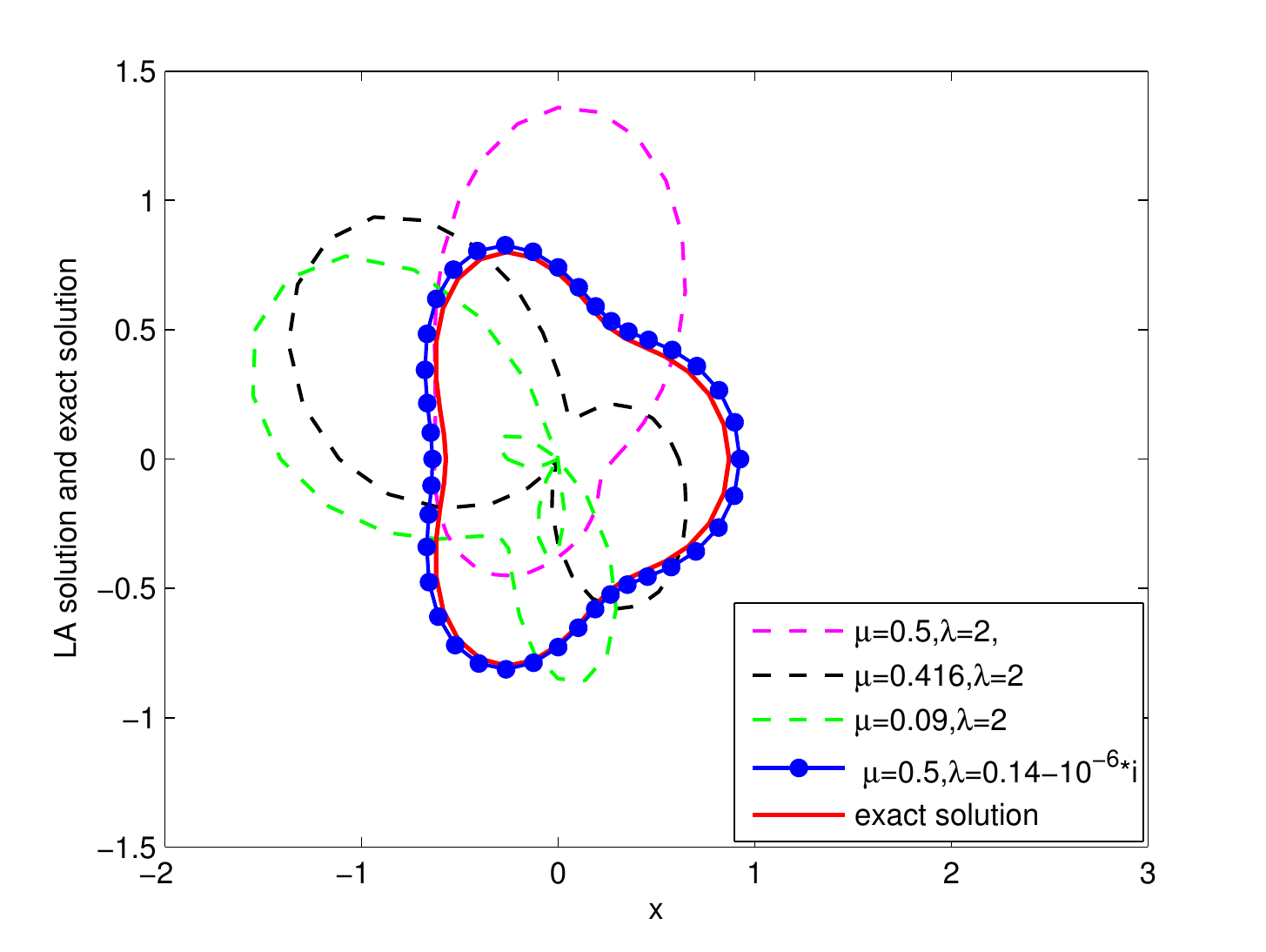}\\
  \caption{Reconstruction the shape  for Example \ref{exm4} with 5\% noise data with various regularization parameter $\mu$.}\label{example4}
\end{figure}

%\begin{figure}
%\centering
%%\subfigure[]{
%    %\label{fig:subfig:a}
%  \includegraphics[width=3in, height=2.5in]{LALQ005b}}
%  \caption{Reconstruction the shape  for Example \ref{exm4} with 5\% noise data with various regularization parameters . }
%  \label{example4}
%\end{figure}

\begin{table}
\small
\centering
\caption{The relative errors for Example \ref{exm4} with different regularization parameters $\mu$ and $\lambda$, under same noise level $\delta=0.05$.}
\vspace{2pt}
\begin{tabular}{|c|c|c|c|c|c|c|c|c|c|c|c|}
\hline
  \multicolumn{1}{|c|}{} &\multicolumn{10}{|c|}{$\lambda=2$}    & \multicolumn{1}{|c|}{$\lambda=0.14-10^{-6}i$}   \\ \cline{1-12}
    $\mu$&$0.1$ &$0.8$ &$ 0.09$ &$0.73$ &$ 0.83$ &$0.206$ &$0.212$ &$0.320$ &$0.416$ &$0.668$ &$0.5$\\%& $ GMsFEM energy error$ & $CEM-GMs energy error$ \\
  \hline
  $e_\gamma$&$0.881$ &$0.720$  &$1.278$ &$0.548$  &$1.096$ &$1.297$ &$1.324$ &$1.034$ &$0.848$ &$0.668$ &$0.0625$ \\
  \hline
\end{tabular}
\label{tab4-nois}
\end{table}
The choice of the regularization parameters is very important in the algorithmic calculation. In Example \ref{exm4}, we fixed the regularization parameter $\mu=0.5$ associated to different values of $\lambda$. In order to eliminate the influence of improper regularization parameters on the inversion results, the regularization parameters can be considered as a random variable with uniform distribution, i.e. $\mu\sim U(0,1)$. In fact, there are many other selection strategies of regularization parameters, such as Morozov¡¯s discrepancy principle, The L-curve method and so on (see \cite{Vogel2002}). But we only concern the comparative effectiveness of the plasmon resonance case and the non-resonance case, and then above simple regularization parameters selection is used. The more advanced selection strategies will be considered in our future works. In Figure \ref{example4}, when $\lambda=2$, we obtain four groups of random numbers from uniform distribution as the regularization parameters, while the corresponding reconstruction results remain disappointing. Table \ref{tab4-nois} lists 10 randomly generated regularization parameters with the relative error also exceeding $50\%$.

\section{Conclusions}
We investigate the inverse problem that utilizing the {far-field} measurement to reconstruct the shape of of an inclusion. The plasmon resonance is proposed to enhance the sensitivity of the reconstruction as well as to reduce the ill-posedness of the inverse problem. In fact, we derive the representation formula of shape sensitivity functional by using the asymptotic expansion method. Based on the asymptotic expansion of the eigenvalues and eigenfunctions of the Neumann-Poincar\'e operator, we further derive the delicate spectral representation of the shape sensitivity functional, which indicates that the sensitivity is improved greatly as the plasmon resonance occurs. Moreover, we combine the Tikhonov regularization method with the Laplace approximation to solve the inverse problem. This hybrid method not only calculates the minimizer accurately and quickly, but also captures the statistical information of the solution.  Finally, extensive numerical experiments confirm our theoretical analysis and illustrate the promising and salient features of the proposed reconstruction scheme.

\section*{Acknowledgments}
The work of M Ding and G Zheng were supported by the NSF of China (11301168) and NSF of Hunan (2020JJ4166). The work of H Liu was supported by the Hong Kong RGC General Research Fund (projects 12302919, 12301218 and 11300821) and the NSFC/RGC Joint Research Grant (project N\_CityU101/21).

\label{ssec:Conclusions}


\begin{thebibliography}{99}

\bibitem{ACL0} {\sc H. Ammari, Y. T. Chow, H. Liu}, {\em Localized sensitivity analysis at high-curvature boundary points of reconstructing inclusions in transmission problems}, SIAM J. Math. Anal., in press, 2022. 

\bibitem{ACL} {\sc H. Ammari, Y. Chow, H. Liu}, {\em Quantum ergodicity and localization of plasmon resonances}, arXiv:2003.03696

\bibitem{ACLS} {\sc H. Ammari, Y. Chow, H. Liu, M. Sunkula}, {\em Quantum integral systems and concentration of plasmon resonance}, arXiv:2109.13008

\bibitem{ACLZ} {\sc H. Ammari, Y. Chow, K. Liu, J. Zou}, {\em Optimal shape design by partial spectral data}, SIAM J. Sci. Comput., 37 (2015) B855-B883.

\bibitem{ACKLM}
{\sc  H. Ammari, G. Ciraolo, H. Kang,  H. Lee, G. Milton}, {\em Spectral theory of a
Neumann-Poincar\'{e}-type operator and analysis of cloaking due to anomalous localized resonance}, Arch. Ration. Mech. Anal. 208 (2013) 667-692.


\bibitem{ADP}
\textcolor[rgb]{0,0,0}{
{\sc H. Ammari, Y. Deng,  P. Millien}, {\em Surface plasmon resonance of nanoparticles and applications in imaging}, Arch. Ration. Mech. Anal., 220 (2016) 109-153.}

%\bibitem{Kang2004}
%{\sc H. Ammari, H. Kang}, {\em Reconstruction of small inhomogeneities from boundary measurements}, Lecture Notes in Mathematics, vol. 1846. Springer, Berlin (2004).


%\bibitem{AmmariKang2007}
%{\sc H. Ammari, H. Kang}, {\em Polarization and Moment Tensors with Applications to Inverse Problems and Effective Medium Theory}, Appl. Math. Sci. 162, Springer, New York, 2007.

%\bibitem{Lim2010}
%{\sc H. Ammari, H. Kang, M. Lim, H. Zribi}, {\em  Conductivity interface problems. Part I: Small
%perturbations of an interface}, Trans. Amer. Math. Soc., 362 (2010) 2435-2449.

\bibitem{Lim2012}
{\sc H. Ammari, H. Kang, M. Lim, H. Zribi}, {\em  The generalized polarization tensors for resolved imaging. Part I: Shape reconstruction of a conductivity inclusion}, Math. Comp., 81 (2012) 367-386.


%\bibitem{Touibi2005}
%{\sc  H. Ammari, H. Kang, K. Touibi}, {\em Boundary layer techniques for deriving the effective properties of composite materials}, Asymp. Anal., 41 (2005) 119-140.

\bibitem{Zhang2017}
{\sc H. Ammari, P. Millien, M. Ruiz, H. Zhang}, {\em Mathematical analysis of plasmonic nanoparticles: the scalar case}, Arch. Ration. Mech. Anal., 224(2) (2017) 597-658.

%\bibitem{Yu2017}
%{\sc H. Ammari, M. Putinar,  M. Ruiz, S. Yu, H. Zhang}, {\em Shape reconstruction of nanoparticles from their associated plasmonic resonances}, J. Math. Pure. Appl., 122 (2019) 23--48.
%
%\bibitem{Ruiz2017}
%{\sc  H. Ammari, M. Ruiz, S. Yu, H. Zhang}, {\em Reconstructing fine details of small objects by using plasmonic spectroscopic data}, SIAM J. Imaging Sci., 11(1) (2018) 1-23.

%\bibitem{ammari20161}
%{\sc H. Ammari, M. Ruiz, S. Yu, H. Zhang}, {\em Mathematical analysis of plasmonic resonances for nanoparticles: The full Maxwell equations}, J. Differ. Equations, 261 (2016), 3615-3669.

%\bibitem{ando2018}
%{\sc K. Ando, Y. Ji, H. Kang, K. Kim, S. Yu}, {\em Spectral properties of the Neumann-Poincar\'{e} operator and cloaking by anomalous localized resonance for the elasto-static system}, European J. Appl. Math., 29 (2018) 189-225.

\bibitem{ando2016}
{\sc K. Ando, H. Kang, H. Liu}, {\em Plasmon resonance with finite frequencies: A validation of the quasi-static approximation for diametrically small inclusions}, SIAM J. Appl. Math., 76 (2016) 731-749.

\bibitem{Anker2010}
{\sc J. Anker, W. Hall, O. Lyandres, N. Shah, J. Zhao, R. Van Duyne}, {\em Biosensing
with plasmonic nanosensors, and Applications}, Cambridge University Press, New York, 2010.

\bibitem{Baffou2010}
{\sc G. Baffou, C. Girard, R. Quidant}, {\em Mapping heat origin in plasmonic structures}, Phys. Rev.
Lett, 104 (2010) 136805.

\bibitem{blasten2020}
{\sc  E. Bl{\aa}sten, H. Li, H. Liu, Y. Wang}, {\em Localization and geometrization in plasmon resonances and geometric structures of Neumann-Poincar\'{e} eigenfunctions}, ESAIM: Math. Model. Numer. Anal., 54 (2020) 957-976.

\bibitem{BL1} {\sc E. Bl{\aa}sten, H. Liu}, {\em On corners scattering stably and stable shape determination by a single far-field pattern}, Indiana Univ. Math. J., 70 (2021), no. 3, 907--947. 

\bibitem{BL2} {\sc  E. Bl{\aa}sten, H. Liu}, {\em Scattering by curvatures, radiationless sources, transmission eigenfunctions, and inverse scattering problems}, SIAM J. Math. Anal. 53 (2021), no. 4, 3801--3837.

\bibitem{BL3} {\sc Bl{\aa}sten, H. Liu}, {\em Recovering piecewise constant refractive indices by a single far-field pattern}, Inverse Problems, 36 (2020), no. 8, 085005, 16 pp.

%\bibitem{bochniak2002}
%{\sc M. Bochniak , F. Cakoni}, {\em Domain sensitivity analysis of the elastic far-field patterns in scattering from nonsmooth obstacles}, J. Math. Anal. Appl., 272 (2002) 318-334.

\bibitem{bouchitte2010}
{\sc G. Bouchitt\'{e}, B. Schweizer}, {\em Cloaking of small objects by anomalous localized resonance}, Quart. J. Mech. Appl. Math., 63 (2010) 437-463.

\bibitem{CDLZ} {\sc X. Cao, H. Diao, H. Liu, J. Zou}, {\em On nodal and generalized singular structures of Laplacian eigenfunctions and applications to inverse scattering problems}, J. Math. Pures Appl. (9), 143 (2020), 116--161. 

\bibitem{CDHLW} {\sc Y. T. Chow, Y. Deng, Y. He, H. Liu, X. Wang}, {\em Surface-localized transmission eigenstates, super-resolution imaging, and pseudo surface plasmon modes}, SIAM J. Imaging Sci., 14 (2021), no. 3, 946--975. 


\bibitem{Banks2009}
{\sc A. Cintr\'on-Arias, H. Banks, A. Capaldi, A. Lloyd}, {\em A sensitivity matrix methodology for inverse problem formulation}, J. Inverse Ill-Pose. P., 17 (2009) 1-20.

\bibitem{DLZ} {\sc Y. Deng, H. Liu, G. Zheng}, {\em Mathematical analysis of plasmon resonances for curved nanorods}, J. Math. Pures Appl. (9), 153 (2021), 248--280.

\bibitem{DCL} {\sc H. Diao, X. Cao, H. Liu}, {\em On the geometric structures of transmission eigenfunctions with a conductive boundary condition and applications}, Comm. Partial Differential Equations, 46 (2021), no. 4, 630--679.

\bibitem{deng2020}
{\sc Y. Deng, H. Li, H. Liu}, {\em Analysis of surface polariton resonance for nanoparticles in elastic system}, SIAM J. Math. Anal., 52 (2020) 1786-1805.

%\bibitem{deng20201}
%{\sc Y. Deng, H. Li, H. Liu}, {\em Spectral properties of Neumann-Poincar\'{e} operator and anomalous localized resonance in elasticity beyond quasi-static limit}, J. Elasticity, 140 (2020) 213-242.

\bibitem{Doicu2010}
{\sc  A. Doicu, T. Trautmann, F. Schreier}, {\em Numerical regularization for atmospheric inverse problems}, Springer Science $\&$ Business Media, 2010.


\bibitem{fang2015}
\textcolor[rgb]{0,0,0}{
{\sc X. Fang, Y. Deng, J. Li}, {\em Plasmon resonance and heat generation in nanostructures}, Math. Method. Appl. Sci., 38(18) (2015) 4663-4672.}

\bibitem{Feng2017}
{\sc  T. Feng,  H. Kang, H. Lee}, {\em  Construction of GPT-vanishing structures using shape derivative}, J. Comput. Math., 35(005) (2017) 569-585.

%\bibitem{fidler2008}
%{\sc  T. Fidler, M. Grasmair, O. Scherzer}, {\em Identifiability and reconstruction of shapes from integral invariants}, Inverse Probl. Imaging, 2 (2008) 341-354.
%
%\bibitem{Friedman1989}
%{\sc  A. Friedman,  M. Vogelius}, {\em Identification of small inhomogeneities of extreme conductivity by boundary measurements: a theorem on continuous dependence}, Arch. Rat. Mech. Anal, 105 (1989) 299-326.

\bibitem{GLWZ} {\sc Y. Gao, H. Liu, X. Wang, K. Zhang}, {\em On an artificial neural network for inverse scattering problems}, J. Comput. Phys. 448 (2022), 110771.

\bibitem{Grieser2014}
{\sc D. Grieser}, {\em The plasmonic eigenvalue problem}, Rev. Math. Phys. 26 (2014) 1450005.

%\bibitem{GH1}
%{\sc R. Griesmaier, M. Hanke}, {\em Multifrequency impedance imaging with multiple signal classification}, SIAM J. Imaging Sci., 8 (2015), 939-967.

%\bibitem{haddar2004}
%{\sc H. Haddar, R. Kress}, {\em On the fr\'{e}chet derivative for obstacle scattering with an impedance boundary condition}, SIAM J. Appl. Math., 65 (2004) 194-208.

\bibitem{H2}
{\sc M. Hanke}, {\it Recent progress in electrical impedance tomography}, Inverse Probl., 19 (2003), S65--S90.

\bibitem{Hanke1997}
{\sc M. Hanke}, {\em A regularizing Levenberg-Marquardt scheme with applications to inverse groundwater filtration problems}, Inverse Probl., 13 (1997) 79.

%\bibitem{Hettlich1998}
%{\sc F. Hettlich, W. Rundell}, {\em The determination of a discontinuity in a conductivity from a single boundary measurement}, Inverse Probl., 14 (1998) 67-82.

\bibitem{hintermuller2015}
{\sc M. Hinterm\"{u}ller, A. Laurain, I. Yousept}, {\em  Shape sensitivities for an inverse problem in magnetic induction tomography based on the eddy current model}, Inverse Probl., 31 (2015) 065006.

%\bibitem{hiptmair2018}
%{\sc R. Hiptmair, J. Li}, {\em Shape derivatives for scattering problems}, Inverse Probl., 34 (2018) 105001.

\bibitem{Iglesias2013}
{\sc  M. Iglesias, K. Law, A. Stuart}, {\em Evaluation of Gaussian approximations for data assimilation in reservoir models}, Comput. Geosci., 17 (2013) 851-885.

%\bibitem{ito2008}
%{\sc K. Ito, K. Kunisch}, {\em Lagrange multiplier approach to variational problems and applications}, Philadelphia, PA: SIAM, 2008.

\bibitem{Jain2006}
{\sc P. Jain, K. Lee, I. El-Sayed, M. El-Sayed}, {\em  Calculated absorption and
scattering properties of gold nanoparticles of different size, shape, and composition:
applications in biomedical imaging and biomedicine}, J. Phys. Chem. B, 110 (2006) 7238-7248.

%\bibitem{kirsch1993}
%{\sc A. Kirsch}, {\em The domain derivative and two applications in inverse scattering theory}, Inverse Probl., 9 (1993) 81-96.

\bibitem{Kress1999}
{\sc R. Kress}, {\em Linear Integral Equations}, 2nd edition, Springer, 1999.

\bibitem{LLL} {\sc H. Li, J. Li, H. Liu}, {\em On quasi-static cloaking due to anomalous localized resonance in $\mathbb{R}^3$}, SIAM J. Appl. Math., 75 (2015), no. 3, 1245--1260. 

\bibitem{li2018}
{\sc H. Li, H. Liu}, {\em On anomalous localized resonance and plasmonic cloaking beyond the quasistatic limit}, Proc. Roy. Soc. A, 474 (2018).

\bibitem{li20181}
{\sc H. Li, J. Li, H. Liu}, {\em On novel elastic structures inducing polariton resonances with finite frequencies and cloaking due to anomalous localized resonances}, J. Math. Pures Appl. (9), 120 (2018) 195-219.

\bibitem{li2019}
{\sc H. Li, S. Li, H. Liu, X. Wang}, {\em Analysis of electromagnetic scattering from plasmonic inclusions beyond the quasi-static approximation and applications}, ESAIM: Math. Model. Numer. Anal., 53 (2019) 1351-1371.

\bibitem{Link2000}
{\sc S. Link, M. El-Sayed}, {\em Shape and size dependence of radiative, non-radiative and
photothermal properties of gold nanocrystals}, Int. Rev. Phys. Chem., 19 (2000) 409-453.

\bibitem{LT} {\sc H. Liu, C. H. Tsou}, {\em Stable determination of polygonal inclusions in Calder\'on's problem by a single partial boundary measurement}, Inverse Problems, 36 (2020), no. 8, 085010, 23 pp.

\bibitem{LTY} {\sc H. Liu, C. H. Tsou and W. Yang}, {\em On Calder\'on's inverse inclusion problem with smooth shapes by a single partial boundary measurement}, Inverse Problems, 37 (2021), no. 5, Paper No. 055005, 18 pp. 

\bibitem{Mayergoyz2005}
{\sc I. Mayergoyz, D. Fredkin, Z. Zhang}, {\em Electrostatic (plasmon) resonances in
nanoparticles}, Phys. Rev. B, 72 (2005) 155412.

\bibitem{WN10}
{\sc G. Milton, N. Nicorovici}, {\em On the cloaking effects associated with anomalous
localized resonance}, Proc. R. Soc. A, {462} (2006) 3027-3059.


\bibitem{Nam2003}
{\sc  J. Nam, C. Thaxton, C. Mirkin}, {\em Nanoparticle-based bio-bar codes for the ultrasensitive detection of proteins}, Science 301 (2003) 1884-886.

\bibitem{nicholls2018}
{\sc D. Nicholls, X. Tong}, {\em A high-order perturbation of surfaces algorithm for the simulation of localized surface plasmon resonances in two dimensions}, J. Sci. Comput. 76 (2018) 1370-1395.

\bibitem{Ordal1983}
{\sc M. Ordal, L. Long, R. Bell, S. Bell, R. Bell, R. Alexander, C. Ward}, {\em Optical properties of the metals al, co, cu, au, fe, pb, ni, pd, pt, ag, ti, and w in the infrared and far infrared}, Appl. Opt., 22 (1983) 1099-1119.

%\bibitem{potthastl1994}
%{\sc R. Potthast}, {\em Fr\'{e}chet differentiability of boundary integral operators in inverse acoustic scattering}, Inverse Probl., 10 (1994) 431-447.


\bibitem{Raschke2003}
{\sc G. Raschke et al. }, {\em  Biomolecular recognition based on single gold nanoparticle light scattering }, Nano Lett., 3(7) (2003) 935-938.

\bibitem{Challener2008}
{\sc D. Sarid, W. Challener}, {\em Modern Introduction to Surface Plasmons: Theory, \textcolor[rgb]{0,0,0}{Mathematical} Modeling}, Nat. Mater., 7 (2008) 442-453.

\bibitem{Schillings2020}
{\sc C. Schillings, B. Sprungk, P. Wacker}, {\em On the Convergence of the Laplace Approximation and
Noise-Level-Robustness of Laplace-based Monte Carlo Methods for Bayesian Inverse Problems}, Numerical Mathematik, 145 (2020) 915-971.


\bibitem{Schultz2000}
{\sc  S. Schultz, D. Smith, J. Mock, D. Schultz}, {\em  Single-target molecule detection with nonbleaching multicolor optical immunolabels}, Proc. Natl Acad. Sci. USA, 97 (2000) 996-1001.

%\bibitem{Grieser D.2014}
%{\sc D. Grieser}, {\em  The plasmonic eigenvalue problem}, Arch. Ration. Mech. Anal, Rev. Math. Phys. 26 (2014) 1450005.
%
%\bibitem{Mayergoyz2006}
%{\sc  I. D. Mayergoyz, Z. Zhang}, {\em Numerical analysis of plasmon resonances in nanopar-
%ticules}, IEEE Trans. Mag., 42 (2006) 759-762.
%
%\bibitem{Khavinson2007}
%{\sc   D. Khavinson, M. Putinar, H. S. Shapiro}, {\em Poincar\'e variational problem in potential theory}, Arch. Ration. Mech. Anal., 185 (2007) 143-184.


%\bibitem{Millien P.2016}
%{\sc  H. Ammari, Y. Deng, P. Millien}, {\em  Surface plasmon resonance of nanoparticles and
%applications in imaging}, Arch. Ration. Mech. Anal., 220 (2016) 109-153.

%\bibitem{Ruiz M.2016}
%{\sc  H. Ammari, M. Ruiz, W. Wu, S. Yu, H. Zhang}, {\em Mathematical and numerical framework for metasurfaces using thin layers of periodically distributed plasmonic nanoparticles}. Proc Math Phys Eng, 472 (2016) 20160445.

\bibitem{Vogel2002}
{\sc C. Vogel},
{\em Computational Methods for Inverse Problems},
SIAM, 2002.

\bibitem{YYL} {\sc W. Yin, W. Yang and H. Liu}, {\em A neural network scheme for recovering scattering obstacles with limited phaseless far-field data}, J. Comput. Phys., 417 (2020), 109594, 18 pp.

\bibitem{Z} {\sc G. Zheng}, {\em Mathematical analysis of plasmonic resonance for 2-D photonic crystal}, J. Differential Equations, 266 (2019), no. 8, 5095--5117.

\end{thebibliography}
\end{document}